\documentclass[reqno]{amsart}
\usepackage[bookmarks=false]{hyperref} 
\usepackage{amssymb}
\usepackage{amsmath}
\usepackage{amsthm} 
\usepackage{mathtools}
\usepackage{color} 
\usepackage[svgnames]{xcolor}
\usepackage[utf8]{inputenc}
\usepackage[T1]{fontenc}
\usepackage{graphicx}
\usepackage{placeins}
\usepackage{vmargin}
\usepackage{setspace}

\hypersetup{
  colorlinks=true,
  linkcolor=Blue  ,          
  citecolor=Red,        
  filecolor=Magenta,      
  urlcolor=Green,           
pdfpagemode=UseOutlines,
pdftitle={},
pdfauthor={Tin Van Phan <van-tin.phan@univ-tlse3.fr>},
pdfsubject={Infinite soliton for derivative nonlinear Schr\"odinger equation},
pdfkeywords={soliton, Schr\"oginer equation} 
}

\numberwithin{equation}{section}
\numberwithin{figure}{section}

\newtheorem{theorem}{Theorem}[section]

\newtheorem{lemma}[theorem]{Lemma}

\newtheorem*{notation}{Notation}
\newtheorem*{assumption a}{Assumption A}
\newtheorem*{assumption b}{Assumption B}
\newtheorem*{assumption c}{Assumption C}

\theoremstyle{definition}

\theoremstyle{remark}
\newtheorem{remark}[theorem]{Remark}

\DeclarePairedDelimiter{\norm}{\lVert}{\rVert}

\newcommand{\N}{\mathbb{N}}
\newcommand{\R}{\mathbb{R}}
\newcommand{\C}{\mathbb C}

\renewcommand{\leq}{\leqslant}
\renewcommand{\geq}{\geqslant}

\DeclareMathAlphabet{\mathpzc}{OT1}{pzc}{m}{it}
\renewcommand{\Re}{\mathcal R\!\mathpzc{e}}
\renewcommand{\Im}{\mathcal I\!\mathpzc{m}}

\begin{document}

\title[Multi-solitons]{Construction of the multi-solitons for a generalized derivative nonlinear Schr\"odinger equation}

\author[Phan Van Tin]{Phan Van Tin}

\address[Phan Van Tin]{Institut de Math\'ematiques de Toulouse ; UMR5219,
  \newline\indent
  Universit\'e de Toulouse ; CNRS,
  \newline\indent
  UPS IMT, F-31062 Toulouse Cedex 9,
  \newline\indent
  France}
\email[Phan Van Tin]{van-tin.phan@univ-tlse3.fr}


\date{\today}

\begin{abstract} 
We consider a derivative nonlinear Schr\"odinger equation with general nonlinearlity:
\begin{equation*}
i\partial_tu+\partial_x^2u+i|u|^{2\sigma}\partial_xu=0,
\end{equation*}
In \cite{TaXu18}, the authors prove the stability of two solitary waves in energy space for $\sigma\in (1,2)$. As a consequence, there exists a solution of the above equation which is close arbitrary to sum of two solitons in energy space when $\sigma\in (1,2)$. Our goal in this paper is proving the existence of multi-solitons in energy space for $\sigma \geq \frac{5}{2}$. Our proofs proceed by fixed point arguments around the desired profile, using Strichartz estimates. 
\end{abstract}

\maketitle
\tableofcontents

\section{Introduction}
In this paper, we consider the following generalized derivative nonlinear Schr\"odinger equation:
\begin{equation}\label{gdnls} 
i\partial_tu+\partial_x^2u+i|u|^{2\sigma}\partial_xu=0,
\end{equation}
where $\sigma\in\R^{+}$ is a given constant and $u:\mathbb R_t\times \mathbb R_x\to \mathbb C$. 

The equation \eqref{gdnls} was studied in many works. In the special case $\sigma=1$, local well-posedness, global well posedness, stability of solitary waves and stability of multi-solitons have been investigated. In \cite{Oz96}, Ozawa gave a sufficient condition for global well posedness of \eqref{gdnls} in the energy space by using a Gauge transformation to remove the derivative terms. In \cite{CoOh06}, Colin-Ohta showed that the equation has a two parameters family of solitary waves and proved the stability of these particular solutions by using variational methods. In \cite{KwWu18}, Kwon-Wu gave a result on stability of solitary waves when the parameters are at the threshold between existence and non-existence. In \cite{CoWu18}, Le Coz-Wu proved stability of multi-solitons in the energy space under some conditions on the parameters of the composing solitons.\\ 
In the general case, the local well-posedness and global well- posedness of \eqref{gdnls} was studied in \cite{HaOz16} when the initial data is in the Sobolev space $H_0^1(\Omega)$, where $\Omega$ is any unbounded interval of $\R$. In this work, Hayashi-Ozawa used an approximation argument. In \cite{Santos15}, Santos proved the local well-posedness for small size initial data in weighted Sobolev spaces. The arguments used in this work follow parabolic regularization approach introduced by Kato \cite{Ka72}. 



The equation \eqref{gdnls} has a two parameters family of solitons. The stability of the solitons has attracted the attention of many researchers. In \cite{LiSiSu13}, by using the abstract theory of Grillakis-Shatah-Strauss \cite{GrShSt87,GrShSt90}, Liu-Simpson-Sulem proved that in the case $\sigma\geq 2$, the solitons of \eqref{gdnls} are orbitally unstable; in the case $0<\sigma<1$, they are orbitally stable and in the case $\sigma \in (1,2)$ they are orbitally stable if $c<2z_0\sqrt{\omega}$ and orbitally unstable if $c>2z_0\sqrt{\omega}$ for some constant $z_0 \in (0,1)$. In the critical case $c=2z_0\sqrt{\omega}$, Guo-Ning-Wu \cite{GuNiWu20} proved that solitons are always orbitally unstable. In \cite{BaWuXu20}, Bai-Wu-Xue proved that when $\sigma\geq 2$, the solution is global and scattering when the initial data small in $H^s(\R)$, $\frac{1}{2}\leq s\leq 1$. Moreover, the authors showed that when $\sigma <2$, the scattering may not occur even under smallness conditions on the initial data. Therefore, in this model, the exponent $\sigma\geq 2$ is optimal for small data scattering. In \cite{TaXu18}, in the case $\sigma\in (1,2)$, Tang and Xu proved the stability of the sum of two solitary waves in the energy space using perturbation arguments, modulational analysis and an energy argument as in \cite{MaMeTs02,MaMeTs06}. In this chapter, we show the existence of multi-solitons in energy space in the case $\sigma \geq \frac{5}{2}$. Before stating the main result, we give some preliminaries on multi-solitons of \eqref{gdnls}. 
 
As mentioned in \cite{LiSiSu13}, the equation \eqref{gdnls} admits a two-parameters family of solitary waves solutions given by
\begin{equation}
\label{formula of soliton}
\psi_{\omega,c}(t,x)=\varphi_{\omega,c}(x-ct)\exp \left( i\left(\omega t+\frac{c}{2}(x-ct)-\frac{1}{2\sigma+2}\int_{-\infty}^{x-ct}\varphi_{\omega,c}^{2\sigma}(\eta)\,d\eta\right)\right),
\end{equation}
where $\omega>\frac{c^2}{4}$ and
\begin{equation}
\label{formula of varphi}
\varphi_{\omega,c}^{2\sigma}(y)=\frac{(\sigma+1)(4\omega-c^2)}{2\sqrt{\omega}\left(\cosh(\sigma\sqrt{4\omega-c^2}y)-\frac{c}{2\sqrt{\omega}}\right)}.
\end{equation}
The profile $\varphi_{\omega,c}$ is a positive solution of 
\begin{equation}
\label{equation of varphi}
-\partial_y^2\varphi_{\omega,c}+\left(\omega-\frac{c^2}{4}\right)\varphi_{\omega,c}+\frac{c}{2}|\varphi_{\omega,c}|^{2\sigma}\varphi_{\omega,c}-\frac{2\sigma+1}{(2\sigma+2)^2}|\varphi_{\omega,c}|^{4\sigma}\varphi_{\omega,c}=0.
\end{equation}
Define
\begin{equation}
\label{define of phi}
\phi_{\omega,c}(y)=\varphi_{\omega,c}(y)e^{i\theta_{\omega,c}(y)},
\end{equation}
where
\begin{equation}
\label{define of theta}
\theta_{\omega,c}(y)=\frac{c}{2}y-\frac{1}{2\sigma+2}\int_{-\infty}^y\varphi_{\omega,c}^{2\sigma}(\eta)\, d\eta.
\end{equation}
Clearly, we have
\begin{equation}
\label{rewrite of psi}
\psi_{\omega,c}(x,t)=e^{i\omega t}\phi_{\omega,c}(x-ct).
\end{equation}
and $\phi_{\omega,c}$ solves
\begin{equation}
\label{equation of phi}
-\partial_y^2\phi_{\omega,c}+\omega\phi_{\omega,c}+ic\partial_y\phi_{\omega,c}-i|\phi_{\omega,c}|^{2\sigma}\partial_y\phi_{\omega,c}=0, \quad y\in\R.
\end{equation}
Let $K\in\N$. For each $1 \leq j \leq K$, let $(\omega_j,c_j,x_j,\theta_j) \in \R^4$ be parameters such that $\omega_j>\frac{c_j^2}{4}$. Define, for each $j=1,...,K$
\[
R_j(t,x)=e^{i\theta_j}\psi_{\omega_j,c_j}(t,x-x_j)
\]
and define the multi-soliton profile by
\begin{equation}
\label{profile of multi soliton}
R=\sum_{j=1}^{K}R_j.
\end{equation}
For convenience, define $h_j=\sqrt{4\omega_j-c_j^2}$, for each $j=1,...,K$. Our main result is the following.
\begin{theorem}
\label{main result}
Let $\sigma\geq \frac{5}{2}$, $K\in \N^{*}$ and for each $1\leq j\leq K$, $(\theta_j,\omega_j,c_j,x_j)$ be a sequence of parameters such that $x_j\in\R$, $\theta_j\in\R$, $c_j\neq c_k$, for $j\neq k$. The multi-soliton profile $R$ is given as in \eqref{profile of multi soliton}. There exists a certain positive constant $C_*$ such that if the parameters $(\omega_j,c_j)$ satisfy
\begin{equation}
\label{condition of existence multi soliton}
C_*\left((1+\norm{R}_{L^{\infty}L^{\infty}}^{2(\sigma-1)})(1+\norm{R}_{L^{\infty}H^1}^2)(1+\norm{\partial_xR}_{L^{\infty}L^{\infty}}+\norm{R}_{L^{\infty}L^{\infty}}^{2\sigma+1})\right) \leq v_{*}=\inf_{j\neq k}h_j|c_j-c_k|,
\end{equation}
then there exists a solution $u$ of \eqref{gdnls} such that
\[
\norm{u-R}_{H^1} \leq Ce^{-\lambda t}, \quad \forall t\geq T_0,
\]
for positive constants $C,T_0$ depending only on the parameters $\omega_1,...,\omega_K,c_1,...,c_K$ and $\lambda=\frac{1}{16}v_{*}$.
\end{theorem}
We have the following comment about the restriction $\sigma\geq\frac{5}{2}$.
\begin{remark}
The following inequality holds for $\sigma\geq 2$:
\begin{equation}\label{eq 11111}
(a+b)^{2(\sigma-2)}-a^{2(\sigma-2)} \lesssim b^{2(\sigma-2)}+ba^{2(\sigma-2)-1}, \quad \text{ for all $a,b>0$. }
\end{equation}
The condition $\sigma\geq \frac{5}{2}$ ensures that the order of $b$ on the right hand side of \eqref{eq 11111} is larger than $1$. This is used in the proof of Lemma \ref{existence slution of system}. 
\end{remark}
The condition \eqref{condition of existence multi soliton} is an implicit condition on the parameters. Below, we show that for large, negative and enough separated velocities, the condition \eqref{condition of existence multi soliton} holds.  
\begin{remark}
We prove that there exist parameters $(\omega_j,c_j,\theta_j,x_j)$ for $1 \leq j\leq K$ such  at the condition \eqref{condition of existence multi soliton} is satisfied. Let $M>0$, $h_j>0$, $d_j<0$, for each $1\leq j\leq K$. We chose $(c_j,\omega_j)=\left(Md_j,\frac{1}{4}(h_j^2+M^2d_j^2)\right)$. We verify that this choice satisfies the condition \eqref{condition of existence multi soliton} for $M$ large enough. Indeed, we see that $c_j<0$ and $h_j \ll |c_j|$ for $M$ large enough. We have
\begin{align*}
\varphi^{2\sigma}_{\omega_j,c_j}&\approx \frac{h_j^2}{2\sqrt{\omega_j}\left(\cosh(\sigma h_j y)-\frac{c_j}{2\sqrt{\omega_j}}\right)}\\
\partial_x\varphi_{\omega_j,c_j}&\approx \left(\frac{h_j^2}{2\sqrt{\omega_j}}\right)^{\frac{1}{2\sigma}}\frac{-\sinh(\sigma h_j y)}{\left(\cosh(\sigma h_j y)-\frac{c_j}{2\sqrt{\omega_j}}\right)^{1+\frac{1}{2\sigma}}}.
\end{align*}
Using $|\sinh(x)| \leq |\cosh(x)|$ for all $x\in\R$ we have
\begin{align*}
|\partial_x\varphi_{\omega_j,c_j}| \leq \left(\frac{h_j^2}{2\sqrt{\omega_j}}\right)^{\frac{1}{2\sigma}}\frac{1}{(\cosh(\sigma h_j y)-\frac{c_j}{2\sqrt{\omega_j}})^{\frac{1}{2\sigma}}}\lesssim |\varphi_{\omega_j,c_j}|.
\end{align*}
Thus,
\begin{align*}
\norm{R_j}_{L^{\infty}L^{\infty}}&=\norm{\varphi_{\omega_j,c_j}}_{L^{\infty}}\lesssim \sqrt[2\sigma]{\frac{h_j^2}{|c_j|}} \ll 1\\
\norm{\partial_xR_j}_{L^{\infty}L^{\infty}}&=\norm{\partial_x\phi_{\omega_j,c_j}}_{L^{\infty}L^{\infty}}\\
&\lesssim \norm{\partial_x\varphi_{\omega_j,c_j}}_{L^{\infty}}+\left\lVert\frac{c_j}{2}\varphi_{\omega_j,c_j}-\frac{1}{2\sigma+2}\varphi_{\omega_j,c_j}^{2\sigma+1}\right\rVert_{L^{\infty}}\\
&\lesssim \norm{\varphi_{\omega_j,c_j}}_{L^{\infty}}+|c_j|\norm{\varphi_{\omega_j,c_j}}_{L^{\infty}}\\
&\lesssim \sqrt[2\sigma]{\frac{h_j^2}{|c_j|}}+|c_j|\sqrt[2\sigma]{\frac{h_j^2}{|c_j|}}.
\end{align*}   
Hence,
\begin{align*}
\norm{R}_{L^{\infty}L^{\infty}}&\lesssim \sum_{j}\sqrt[2\sigma]{\frac{h_j^2}{|c_j|}} \lesssim 1\\
\norm{\partial_xR}_{L^{\infty}L^{\infty}}&\lesssim \sum_j \left(\sqrt[2\sigma]{\frac{h_j^2}{|c_j|}}+|c_j|\sqrt[2\sigma]{\frac{h_j^2}{|c_j|}}\right).
\end{align*}
Furthermore, 
\begin{align*}
\norm{R_j}_{L^{\infty}H^1}^2&=\norm{R_j}_{L^{\infty}L^2}^2+\norm{\partial_xR_j}^2_{L^{\infty}L^2}=\norm{\varphi_{\omega_j,c_j}}_{L^2}^2+\norm{\partial_x\varphi_{\omega_j,c_j}}_{L^2}^2\\
&\lesssim  \norm{\varphi_{\omega_j,c_j}}_{L^2}^2 \lesssim \left(\frac{h_j^2}{2\sqrt{\omega_j}}\right)^{\frac{1}{\sigma}}\left\lVert\frac{1}{\cosh(\sigma h_j y)^{\frac{1}{2\sigma}}}\right\rVert^2_{L^2}\lesssim \left(\frac{h_j^2}{2\sqrt{\omega_j}}\right)^{\frac{1}{\sigma}}\norm{e^{-\frac{h_j}{2}|y|}}^2_{L^2}\\
&\approx \left(\frac{h_j^2}{2\sqrt{\omega_j}}\right)^{\frac{1}{\sigma}} \frac{1}{h_j}\lesssim h_j^{\frac{1}{\sigma}}h_j^{-1}=h_j^{\frac{1}{\sigma}-1},
\end{align*}
where we use $h_j \leq 2\sqrt{\omega_j}$. Thus,
\begin{align*}
\norm{R}_{L^{\infty}H^1}^2&\lesssim \sum_jh_j^{\frac{1}{\sigma}-1}.
\end{align*}
The condition \eqref{condition of existence multi soliton} satisfies if the following estimate holds:
\begin{align}\label{estimate final of final}
C_*\left(\left(1+\sum_jh_j^{\frac{1}{\sigma}-1}\right)\left(1+\sum_j \left(\sqrt[2\sigma]{\frac{h_j^2}{|c_j|}}+|c_j|\sqrt[2\sigma]{\frac{h_j^2}{|c_j|}}\right)\right)\right) \leq \inf_{j\neq k}h_j|c_j-c_k|.
\end{align}
We see that the left hand side of \eqref{estimate final of final} is order $M^{1-\frac{1}{2\sigma}}$ and the right hand side of \eqref{estimate final of final} is order $M^1$. Hence, the condition \eqref{condition of existence multi soliton} satisfies if we choose $M$ large enough.  
\end{remark}

Our strategy of the proof of Theorem \ref{main result} is as follows. First, we define $\varphi,\psi$ based on $u$ in such a way that $\varphi$ and $\psi$ satisfy a system of nonlinear Schr\"odinger equations without derivatives (see \eqref{system equation}). Let $R$ be a multi-soliton profile which satisfies the assumptions of Theorem \ref{main result}. Then $R$ solves \eqref{gdnls} up to a small perturbation. Let $(h,k)$ be defined in a similar way as $(\varphi,\psi)$ but replace $u$ by $R$. We see that $(h,k)$ solves \eqref{system equation} up to small perturbations. Setting $\tilde{\varphi}=\varphi-h$ and $\tilde{\psi}=\psi-k$, we see that if $u$ solves \eqref{gdnls} then $(\tilde{\varphi},\tilde{\psi})$ solves a system and a relation between $\tilde{\varphi}$ and $\tilde{\psi}$ holds and vice versa. By using the Banach fixed point theorem, we prove that there exists a  solution $(\tilde{\varphi},\tilde{\psi})$ of this system which exponential decays in time on $H^1(\R)$ for $t$ large. Combining with the assumption \eqref{condition of existence multi soliton}, we can prove a relation between $\tilde{\varphi}$ and $\tilde{\psi}$. Thus, we easily obtain the solution $u$ of \eqref{gdnls} satisfying the desired property.  

This chapter is organized as follows. In Section \ref{section 51}, we prove the existence of multi-solitons for the equation \eqref{gdnls}. In Section \ref{section52}, we prove some technical results which are used in the proof of the main result Theorem \ref{main result}. More precisely, we prove the exponential decay of perturbations in the equations of $h,k$ (Lemma \ref{estimate on profile}) and the existence of decaying solutions for the system of equations of $\tilde{\varphi},\tilde{\psi}$ (Lemma \ref{existence slution of system}).   

Before proving the main result, we introduce some notation used in this chapter.
\begin{notation}\label{notation of the paper}
\item[(1)] We denote the Schr\"odinger operator as follows
\[
L=i\partial_t+\partial^2_x.
\]
\item[(2)] Given a time $t\in\R$, the Strichartz space $S([t,\infty))$ is defined via the norm
\[
\norm{u}_{S([t,\infty))}=\sup_{(q,r) \text{ admissible }}\norm{u}_{L^q_t L^r_x([t,\infty)\times\R)}.
\] 
We denote the dual space by $N[t,\infty)=S([t,\infty))^{*}$. Hence for any $(q,r)$ admissible pair we have
\[
\norm{u}_{N([t,\infty))}\leq \norm{u}_{L^{q'}_tL^{r'}_x([t,\infty)\times\R)}.
\]
\item[(3)] For $a,b \in \R^2$, we denote $|(a,b)|=|a|+|b|$.
\item[(4)] Let $a,b>0$. We denote $a\lesssim b$ if $a$ is smaller than $b$ up to multiplication by a positive constant and denote $a \lesssim_c b$ if $a$ is smaller than $b$ up to multiplication by a positive constant depending on $c$. Moreover, we denote $a \approx b$ if $a$ equals to $b$ up to multiplication by a positive constant.  
\end{notation}

\section{Proof of the main result}\label{section 51}
In this section we give the proof of Theorem \ref{main result}. We use the Banach fixed point theorem and Strichartz estimates. We divide our proof in three steps.\\
\textbf{Step 1. Preliminary analysis.}
Let $u \in C(I,H^1(\R))$ be a $H^1(\R)$ solution of \eqref{gdnls} on $I$. Consider the following transform:
\begin{align}
\varphi(t,x)&=\exp(i\Lambda)u(t,x), \label{eqof varphi}\\
\psi&= \exp(i\Lambda)\partial_xu=\partial_x\varphi-\frac{i}{2}|\varphi|^{2\sigma}\varphi,\label{eqof psi}
\end{align} 
where 
\[
\Lambda=\frac{1}{2}\int_{-\infty}^x|u(t,y)|^{2\sigma}\,dy.
\]
As in \cite[section 4]{HaOz16}, using $|u|=|\varphi|$ and  $\Im(\overline{u}\partial_xu)=\Im(\overline{\varphi}\psi)$, we have
\begin{align*}
\partial_t\Lambda&= -\sigma\Im(|u|^{2(\sigma-1)}\overline{u}\partial_xu)+\sigma\Im\left[\int_{-\infty}^x\partial_x(|u|^{2(\sigma-1)}\overline{u})\partial_xu\,dy\right]-\frac{1}{4}|u|^{4\sigma}.
\end{align*}
Thus, using $|u|=|\varphi|$ and  $\Im(\overline{u}\partial_xu)=\Im(\overline{\varphi}\psi)$, we have
\begin{align*}
\partial_t\Lambda&=-\sigma|\varphi|^{2(\sigma-1)}\Im(\overline{\varphi}\psi)+\sigma\int_{-\infty}^x\partial_x(|u|^{2(\sigma-1)})\Im(\overline{u}\partial_xu)\,dx-\frac{1}{4}|\varphi|^{4\sigma}\\
&=-\sigma|\varphi|^{2(\sigma-1)}\Im(\overline{\varphi}\psi)+\sigma\int_{-\infty}^x\partial_x(|\varphi|^{2(\sigma-1)})\Im(\overline{\varphi}\psi)\,dx-\frac{1}{4}|\varphi|^{4\sigma}.
\end{align*} 
Since $u$ solves \eqref{gdnls}, we have
\begin{align*}
L\varphi&= L(\exp(i\Lambda))u+\exp(i\Lambda)Lu+2\partial_x(\exp(i\Lambda))\partial_xu\\
&=L(\exp(i\Lambda))u+\exp(i\Lambda)(Lu+i|u|^{2\sigma}u)\\
&=L(\exp(i\Lambda))u\\
&=(i\partial_t+\partial_x^2)(\exp(i\Lambda))u,\\
&=\left[-\exp(i\Lambda)\partial_t\Lambda +\partial_x(\exp(i\Lambda)\frac{i}{2}|u|^{2\sigma})\right]u\\
&=-\varphi\partial_t\Lambda+\left[\exp(i\Lambda)\frac{-1}{4}|u|^{2\sigma}+\frac{i}{2}\exp(i\Lambda)\partial_x(|u|^{2\sigma})\right]u\\
&=-\varphi\partial_t\Lambda+\varphi\left[-\frac{1}{4}|\varphi|^{4\sigma}+\frac{i}{2}\partial_x(|\varphi|^{2\sigma})\right]\\
&=\sigma|\varphi|^{2(\sigma-1)}\varphi\Im(\overline{\varphi}\psi)-\sigma\varphi\int_{-\infty}^x\partial_x(|\varphi|^{2(\sigma-1)})\Im(\overline{\varphi}\psi)\,dx\\
&\quad+\frac{1}{4}|\varphi|^{4\sigma}\varphi-\frac{1}{4}\varphi|\varphi|^{4\sigma}+i\sigma|\varphi|^{2(\sigma-1)}\varphi\Re(\overline{\varphi}\partial_x\varphi)\\
&=\sigma|\varphi|^{2(\sigma-1)}\varphi(\Im(\overline{\varphi}\psi)+i\Re(\overline{\varphi}\partial_x\varphi))\\
&\quad -\sigma\varphi\int_{-\infty}^x|\varphi|^{2(\sigma-2)}(\sigma-1)\partial_x(|\varphi|^2)\Im(\overline{\varphi}\psi)\,dx\\
&=\sigma|\varphi|^{2(\sigma-1)}\varphi(\Im(\overline{\varphi}\psi)+i\Re(\overline{\varphi}\psi))\\
&\quad-\sigma(\sigma-1)\varphi\int_{-\infty}^x|\varphi|^{2(\sigma-2)}2\Re(\overline{\varphi}\psi)\Im(\overline{\varphi}\psi)\,dx\\
&=i\sigma|\varphi|^{2(\sigma-1)}\varphi^2\overline{\psi}-\sigma(\sigma-1)\varphi\int_{-\infty}^x|\varphi|^{2(\sigma-2)}\Im(\psi^2\overline{\varphi}^2)\,dy.
\end{align*}
As in \cite[section 4]{HaOz16}, we have
\begin{align*}
L\psi&=L(\exp(i\Lambda)\partial_xu)\\
&=\exp(i\Lambda)\left[-\frac{i}{2}\partial_x(|u|^{2\sigma})\partial_xu+\sigma|u|^{2(\sigma-1)}\Im(\overline{u}\partial_xu)\partial_xu\right.\\
&\quad \left.-\sigma\int_{-\infty}^x\Im(\partial_x(|u|^{2(\sigma-1)}\overline{u})\partial_xu)\,dy\partial_xu\right]\\
&=-\frac{i}{2}\partial_x(|\varphi|^{2\sigma})\psi+\sigma|\varphi|^{2(\sigma-1)}\Im(\overline{\varphi}\psi)\psi-\sigma\int_{-\infty}^x\partial_x(|u|^{2(\sigma-1)}) \Im(\overline{u}\partial_xu)\,dy\psi\\
&=-\frac{i}{2}\partial_x(|\varphi|^{2\sigma})\psi+\sigma|\varphi|^{2(\sigma-1)}\psi\Im(\overline{\varphi}\psi)-\sigma\psi\int_{-\infty}^x\partial_x(|\varphi|^{2(\sigma-1)})\Im(\overline{\varphi}\psi)\,dy\\
&=\sigma|\varphi|^{2(\sigma-1)}\psi(\Im(\overline{\varphi}\psi)-i\Re(\overline{\varphi}\partial_x\varphi))\\
&\quad -\sigma\psi\int_{-\infty}^x(\sigma-1)|\varphi|^{2(\sigma-1)}2\Re(\overline{\varphi}\partial\varphi)\Im(\overline{\varphi}\psi)\,dy\\
&=\sigma|\varphi|^{2(\sigma-1)}\psi(\Im(\overline{\varphi}\psi)-i\Re(\overline{\varphi}\psi))\\
&\quad-\sigma(\sigma-1)\psi\int_{-\infty}^x|\varphi|^{2(\sigma-2)}2\Re(\overline{\varphi}\psi)\Im(\overline{\varphi}\psi)\Im(\overline{\varphi}\psi)\,dy\\
&=-i\sigma |\varphi|^{2(\sigma-1)}\psi^2\overline{\varphi}-\sigma(\sigma-1)\psi\int_{-\infty}^x|\varphi|^{2(\sigma-2)}\Im(\psi^2\overline{\varphi}^2)\,dy. 
\end{align*}
Thus, if $u$ solves \eqref{gdnls} then $(\varphi,\psi)$ solves 
\begin{equation}
\label{system equation}
\begin{cases}
L\varphi&=i\sigma|\varphi|^{2(\sigma-1)}\varphi^2\overline{\psi}-\sigma(\sigma-1)\varphi\int_{-\infty}^x|\varphi|^{2(\sigma-2)}\Im(\psi^2\overline{\varphi}^2)\,dy,\\
L\psi&=-i\sigma |\varphi|^{2(\sigma-1)}\psi^2\overline{\varphi}-\sigma(\sigma-1)\psi\int_{-\infty}^x|\varphi|^{2(\sigma-2)}\Im(\psi^2\overline{\varphi}^2)\,dy.
\end{cases}
\end{equation}
For convenience, we define
\begin{align}
P(\varphi,\psi)&=i\sigma|\varphi|^{2(\sigma-1)}\varphi^2\overline{\psi}-\sigma(\sigma-1)\varphi\int_{-\infty}^x|\varphi|^{2(\sigma-2)}\Im(\psi^2\overline{\varphi}^2)\label{eq of P},\\
Q(\varphi,\psi)&=-i\sigma |\varphi|^{2(\sigma-1)}\psi^2\overline{\varphi}-\sigma(\sigma-1)\psi\int_{-\infty}^x|\varphi|^{2(\sigma-2)}\Im(\psi^2\overline{\varphi}^2).\label{eq of Q}
\end{align} 
Let $R$ be the multi-soliton profile satisfying the assumption of Theorem \ref{main result}. Define $h,k$ by
\begin{align*}
h(t,x)&=\exp\left(\frac{i}{2}\int_{-\infty}^x|R(t,x)|^{2\sigma}\,dy\right)R(t,x),\\
k&=\partial_xh-\frac{i}{2}|h|^{2\sigma}h.
\end{align*}
Since $R_j$ solves \eqref{gdnls} for each $1\leq j\leq K$, we have
\begin{equation}
\label{equation of profile of multi soliton before}
LR+i|R|^{2\sigma}R_x=-\sum_{j}i|R_j|^{2\sigma}R_{jx}+i|R|^{2\sigma}R_x.
\end{equation}
By Lemma \ref{estimate on profile} for $t \gg T_0$ large enough we have
\begin{align}\label{estimate need to prove}
\left\lVert-\sum_{j}i|R_j|^{2\sigma}R_{jx}+i|R|^{2\sigma}R_x\right\rVert_{H^2} &\leq e^{-\lambda t}.
\end{align}
Thus, we rewrite \eqref{equation of profile of multi soliton before} as follows:
\begin{equation}
\label{equation of profile of multi soliton}
LR+i|R|^{2\sigma}R_x=e^{-\lambda t}v,
\end{equation}
where
\begin{equation}
\label{equation of v}
v=e^{\lambda t}(-\sum_{j}i|R_j|^{2\sigma}R_{jx}+i|R|^{2\sigma}R_x).
\end{equation}
By an elementary calculation, we have
\begin{equation}
\label{system h k}
\begin{cases}
Lh=i\sigma|h|^{2(\sigma-1)}h^2\overline{k}-\sigma(\sigma-1)h\int_{-\infty}^x|h|^{2(\sigma-2)}\Im(k^2\overline{h}^2)\,dy+e^{-\lambda t}m(t,x),\\
Lk=-i\sigma|h|^{2(\sigma-1)}k^2\overline{h}-\sigma(\sigma-1)k\int_{-\infty}^x|h|^{2(\sigma-2)}\Im(k^2\overline{h}^2)\,dy+e^{-\lambda t}n(t,x).
\end{cases}
\end{equation}
where 
\begin{align}
m&=\exp\left(\frac{i}{2}\int_{-\infty}^x|R|^{2\sigma}\,dy\right)v-\sigma h\int_{-\infty}^x|R|^{2(\sigma-1)}\Im(\overline{R}v)\,dy, \label{equation of m}\\
n&=\exp\left(\frac{i}{2}\int_{-\infty}^x|R|^{2\sigma}\,dy\right)e^{-\lambda t}(\partial_x v-\sigma\partial_xR\int_{-\infty}^x|R|^{2(\sigma-1)}\Im(\overline{R}v)\,dy).\label{equation of n}
\end{align}
Since $v$ is uniformly bounded in time in $H^2(\R)$, we see that $m,n$ are uniformly bounded in time in $H^1(\R)$. Let $\tilde{\varphi}=\varphi-h$ and $\tilde{\psi}=\psi-k$. Then $(\tilde{\varphi},\tilde{\psi})$ solves:
\begin{equation}
\label{system new}
\begin{cases}
L\tilde{\varphi}=P(\varphi,\psi)-P(h,k)-e^{-\lambda t}m(t,x),\\
L\tilde{\psi}=Q(\varphi,\psi)-Q(h,k)-e^{-\lambda t}n(t,x).
\end{cases}
\end{equation}
Set $\eta=(\tilde{\varphi},\tilde{\psi})$, $W=(h,k)$ and $f(\varphi,\psi)=(P(\varphi,\psi),Q(\varphi,\psi)$ and $H=e^{-\lambda t}(m,n)$. We find a solutions of \eqref{system new} in Duhamel form:
\begin{equation}
\label{equation of eta}
\eta(t)=i\int_t^{\infty}[f(W+\eta)-f(W)+H](s)\,ds,
\end{equation}
where $S(t)$ denote the Schr\"odinger group. Moreover, since $\psi=\partial_x\varphi-\frac{i}{2}|\varphi|^{2\sigma}\varphi$, we have
\begin{equation}
\label{relation of tilde psi and tilde varphi}
\tilde{\psi}=\partial_x\tilde{\varphi}-\frac{i}{2}(|\tilde{\varphi}+h|^{2\sigma}(\tilde{\varphi}+h)-|h|^{2\sigma}h).
\end{equation}
\textbf{Step 2. Existence of a solution of the system}\\
From Lemma \ref{existence slution of system}, there exists $T_{*} \gg 1$ such that for $T_0 \geq T_{*}$ there exists a unique solution $\eta$ of \eqref{system new} defined on $[T_0,T_{*})$ such that 
\begin{equation}
\label{estimate on solution of system}
\norm{\eta}_X:=e^{\lambda t}\norm{\eta}_{S([t,\infty)) \times S([t,\infty))}+e^{\lambda t}\norm{\partial_x\eta}_{S([t,\infty))\times S([t,\infty))}\leq 1 \quad \forall t\geq T_0.
\end{equation}
Thus, for all $t \geq T_0$, we have
\begin{equation}
\label{estimate on H1 of solution of system}
\norm{\tilde{\varphi}}_{H^1}+\norm{\tilde{\psi}}_{H^1}\lesssim e^{-\lambda t}.
\end{equation}
\textbf{Step 3. Existence of a multi-soliton train}

We prove that the solution $\eta=(\tilde{\varphi},\tilde{\psi})$ of \eqref{system new} satisfies the relation \eqref{relation of tilde psi and tilde varphi}. Set $\varphi=\tilde{\varphi}+h$, $\psi=\tilde{\psi}+k$ and $v=\partial_x\varphi-\frac{i}{2}|\varphi|^2\varphi$ and $\tilde{v}=v-k$. Since $(\tilde{\varphi},\tilde{\psi})$ solves \eqref{system new} and $(h,k)$ solves \eqref{system h k}, we have $(\varphi,\psi)$ solves \eqref{system equation}. Furthermore,
\begin{align}
Lv&=\partial_xL\varphi-\frac{i}{2}L(|\varphi|^{2\sigma}\varphi). \label{eq 1000}
\end{align}
Moreover,
\begin{align*}
&L(|\varphi|^{2\sigma}\varphi)\\
&=(i\partial_t+\partial_x^2)(\varphi^{\sigma+1}\overline{\varphi}^{\sigma})=i\partial_t(\varphi^{\sigma+1}\overline{\varphi}^{\sigma})+\partial_x^2(\varphi^{\sigma+1}\overline{\varphi}^{\sigma})\\
&=i(\sigma+1)|\varphi|^{2\sigma}\partial_t\varphi+i\sigma|\varphi|^{2(\sigma-1)}\varphi^2\partial_t\overline{\varphi}\\
&\quad +\partial_x((\sigma+1)|\varphi|^{2\sigma}\partial_x\varphi+\sigma |\varphi|^{2(\sigma-1)}\varphi^2\partial_x\overline{\varphi})\\
&=i(\sigma+1)|\varphi|^{2\sigma}\partial_t\varphi+i\sigma|\varphi|^{2(\sigma-1)}\varphi^2\partial_t\overline{\varphi}+(\sigma+1)\left[\partial_x^2\varphi|\varphi|^{2\sigma}+\partial_x\varphi\partial_x(|\varphi|^{2\sigma})\right]\\
&\quad+\sigma\left[\partial_x^2\overline{\varphi}|\varphi|^{2(\sigma-1)}\varphi^2+(\sigma+1)|\partial_x\varphi|^2|\varphi|^{2(\sigma-1)}\varphi+(\sigma-1)|\varphi|^{2(\sigma-2)}\varphi^3(\partial_x\overline{\varphi})^2\right]\\
&=(\sigma+1)|\varphi|^{2\sigma}(i\partial_t\varphi+\partial_x^2\varphi)+\sigma|\varphi|^{2(\sigma-1)}\varphi^2(i\partial_t\overline{\varphi}+\partial_x^2\overline{\varphi})+(\sigma+1)\partial_x\varphi\partial_x(|\varphi|^{2\sigma})\\
&\quad+\sigma(\sigma+1)|\partial_x\varphi|^2|\varphi|^{2(\sigma-1)}\varphi+\sigma(\sigma-1)(\partial_x\overline{\varphi})^2|\varphi|^{2(\sigma-2)}\varphi^3\\
&=(\sigma+1)|\varphi|^{2\sigma}L\varphi+\sigma|\varphi|^{2(\sigma-1)}\varphi^2(-\overline{L\varphi}+2\partial_x^2\overline{\varphi})+(\sigma+1)\partial_x\varphi\partial_x(|\varphi|^{2\sigma})\\
&\quad+\sigma(\sigma+1)|\partial_x\varphi|^2|\varphi|^{2(\sigma-1)}\varphi+\sigma(\sigma-1)(\partial_x\overline{\varphi})^2|\varphi|^{2(\sigma-2)}\varphi^3.
\end{align*}
Combining with \eqref{eq 1000} and using \eqref{system equation}, we have
\begin{align*}
Lv&=\partial_xL\varphi-\frac{i}{2}L(|\varphi|^{2\sigma}\varphi)\\
&=\partial_xL\varphi-\frac{i}{2}\left[(\sigma+1)|\varphi|^{2\sigma}L\varphi+\sigma|\varphi|^{2(\sigma-1)}\varphi^2(-\overline{L\varphi}+2\partial_x^2\overline{\varphi})\right.\\
&\left. +(\sigma+1)\partial_x\varphi\partial_x(|\varphi|^{2\sigma})+\sigma(\sigma+1)|\partial_x\varphi|^2|\varphi|^{2(\sigma-1)}\varphi+\sigma(\sigma-1)(\partial_x\overline{\varphi})^2|\varphi|^{2(\sigma-2)}\varphi^3\right]\\
&=\partial_x(P(\varphi,\psi)-P(\varphi,v))+\partial_xP(\varphi,v)-\frac{i}{2}(\sigma+1)|\varphi|^{2\sigma}\left(P(\varphi,\psi)-P(\varphi,v)\right)\\
&\quad-\frac{i}{2}(\sigma+1)|\varphi|^{2\sigma}P(\varphi,v)+\frac{i}{2}\sigma|\varphi|^{2(\sigma-1)}\varphi^2(\overline{P(\varphi,\psi)}-\overline{P(\varphi,v)})\\
&\quad+\frac{i}{2}\sigma|\varphi|^{2(\sigma-1)}\varphi^2\overline{P(\varphi,v)}-i\sigma|\varphi|^{2(\sigma-1)}\varphi^2\partial_x^2\overline{\varphi}\\
&\quad-\frac{i}{2}\left[(\sigma+1)\partial_x\varphi\partial_x(|\varphi|^{2\sigma})+\sigma(\sigma+1)|\partial_x\varphi|^2|\varphi|^{2(\sigma-1)}\varphi \right.\\
&\quad\left. +\sigma(\sigma-1)(\partial_x\overline{\varphi})^2|\varphi|^{2(\sigma-2)}\varphi^3\right]\\
&=\partial_x(P(\varphi,\psi)-P(\varphi,v))-\frac{i}{2}(\sigma+1)|\varphi|^{2\sigma}\left(P(\varphi,\psi)-P(\varphi,v)\right)\\
&\quad+\frac{i}{2}\sigma|\varphi|^{2(\sigma-1)}\varphi^2(\overline{P(\varphi,\psi)}-\overline{P(\varphi,v)})+G(\varphi,v),
\end{align*}
where $G(\varphi,v)$ contains the remaining ingredients and $G(\varphi,v)$ only depends on $\varphi$ and $v$:
\begin{align}
&G(\varphi,v)\nonumber\\
&=\partial_xP(\varphi,v)-\frac{i}{2}(\sigma+1)|\varphi|^{2\sigma}P(\varphi,v)+\frac{i}{2}\sigma|\varphi|^{2(\sigma-1)}\varphi^2\overline{P(\varphi,v)}\nonumber\\
&\quad-i\sigma|\varphi|^{2(\sigma-1)}\varphi^2\partial_x^2\overline{\varphi}-\frac{i}{2}\left[(\sigma+1)\partial_x\varphi\partial_x(|\varphi|^{2\sigma})+\sigma(\sigma+1)|\partial_x\varphi|^2|\varphi|^{2(\sigma-1)}\varphi\right.\nonumber\\
&\quad\left.+\sigma(\sigma-1)(\partial_x\overline{\varphi})^2|\varphi|^{2(\sigma-2)}\varphi^3\right]. \label{eq of G(varphi,v)}
\end{align}
As the calculations of $L\psi$ in the step 1, noting that the role of $v$ is similar to the role of $\psi$ in the process of calculation, we have $G(\varphi,v)=Q(\varphi,v)$ (see Lemma \ref{equality of G and Q} for a detailed proof). Hence, 
\begin{align*}
L\psi-Lv&=Q(\varphi,\psi)-Q(\varphi,v)-\partial_x(P(\varphi,\psi)-P(\varphi,v))\\
&\quad+\frac{i}{2}(\sigma+1)|\varphi|^{2\sigma}\left(P(\varphi,\psi)-P(\varphi,v)\right)\\
&\quad-\frac{i}{2}\sigma|\varphi|^{2(\sigma-1)}\varphi^2(\overline{P(\varphi,\psi)}-\overline{P(\varphi,v)}).
\end{align*}
Thus,
\begin{align}
L\tilde{\psi}-L\tilde{v}&=L\psi-Lv\nonumber\\
&=Q(\varphi,\tilde{\psi}+k)-Q(\varphi,\tilde{v}+k)-\partial_x(P(\varphi,\tilde{\psi}+k)-P(\varphi,\tilde{v}+k)\nonumber\\
&\quad+\frac{i}{2}(\sigma+1)|\varphi|^{2\sigma}(P(\varphi,\tilde{\psi}+k)-P(\varphi,\tilde{v}+k))\nonumber\\
&\quad-\frac{i}{2}\sigma|\varphi|^{2(\sigma-1)}\varphi^2(\overline{P(\varphi,\tilde{\psi}+k)}-\overline{P(\varphi,\tilde{v}+k)}).\label{eq 111}
\end{align}
Multiplying both side of \eqref{eq 111} by $\overline{\tilde\psi-\tilde{v}}$, taking imaginary part and integrating over space with integration by parts we obtain
\begin{align}
&\frac{1}{2}\partial_t\norm{\tilde{\psi}-\tilde{v}}^2_{L^2}\nonumber\\
&= \Im\int_{\R}(Q(\varphi,\tilde{\psi}+k)-Q(\varphi,\tilde{v}+k))(\overline{\tilde{\psi}}-\overline{\tilde{v}})\,dx\label{term1q}\\
&\quad -\Im\int_{\R}\partial_x(P(\varphi,\tilde{\psi}+k)-P(\varphi,\tilde{v}+k))(\overline{\tilde{\psi}}-\overline{\tilde{v}})\,dx\label{term2w}\\
&\quad +(\sigma+1)\Im \int_{\R}\frac{i}{2}|\varphi|^{2\sigma}(P(\varphi,\tilde{\psi}+k)-P(\varphi,\tilde{v}+k))(\overline{\tilde{\psi}}-\overline{\tilde{v}})\,dx\label{term3t}\\
&\quad -\sigma\Im\int_{\R}\frac{i}{2}|\varphi|^{2(\sigma-1)}\varphi^2(\overline{P(\varphi,\tilde{\psi}+k)}-\overline{P(\varphi,\tilde{v}+k)})(\overline{\tilde{\psi}}-\overline{\tilde{v}})\,dx.\label{term4m}
\end{align}
We denote by $A,B,C,D$ the terms \eqref{term1q}, \eqref{term2w}, \eqref{term3t} and \eqref{term4m} respectively. First, we try to estimate $A,B,C,D$ in term of $R$. We have 
\begin{align}
|A|&\lesssim\left|\int_{\R}(Q(\varphi,\tilde{\psi}+k)-Q(\varphi,\tilde{v}+k))(\overline{\tilde{\psi}}-\overline{\tilde{v}})\,dx\right|\nonumber\\
&\lesssim \left|\int_{\R}|\varphi|^{2(\sigma-1)}\overline{\varphi}((\tilde{\psi}+k)^2-(\tilde{v}+k)^2)(\overline{\tilde{\psi}}-\overline{\tilde{v}})\,dx\right|\nonumber\\
&\quad+\left|\int_{\R}\left[(\tilde{\psi}+k)\int_{-\infty}^x|\varphi|^{2(\sigma-2)}\Im((\tilde{\psi}+k)^2\overline{\varphi}^2)\,dy\right.\right.\nonumber\\
&\quad \left. \left.-(\tilde{v}+k)\int_{-\infty}^x|\varphi|^{2(\sigma-2)}\Im((\tilde{v}+k)^2\overline{\varphi}^2)\,dy\right]  (\overline{\tilde{\psi}}-\overline{\tilde{v}}) \,dx\right|\nonumber\\
&\lesssim \left|\int_{\R}|\varphi|^{2(\sigma-1)}\overline{\varphi}((\tilde{\psi}+k)^2-(\tilde{v}+k)^2)(\overline{\tilde{\psi}}-\overline{\tilde{v}})\,dx\right|\nonumber\\
&\quad +\left|\int_{\R}\left[(\tilde{\psi}-\tilde{v})\int_{-\infty}^x|\varphi|^{2(\sigma-2)}\Im((\tilde{\psi}+k)^2\overline{\varphi}^2)\,dy\right](\overline{\tilde{\psi}}-\overline{\tilde{v}})\,dx\right|\nonumber\\
&\quad+\left|\int_{\R}\left[(\tilde{v}+k)\int_{-\infty}^x|\varphi|^{2(\sigma-2)}\Im(\overline{\varphi}^2((\tilde{\psi}+k)^2-(\tilde{v}+k)^2))\,dy\right](\overline{\tilde{\psi}}-\overline{\tilde{v}})\,dx\right|\nonumber\\
&\lesssim \norm{\tilde{\psi}-\tilde{v}}^2_{L^2}\norm{\varphi}^{2\sigma-1}_{L^{\infty}}\norm{\tilde{\psi}+\tilde{v}+2k}_{L^{\infty}}\nonumber\\
&\quad+\norm{\tilde{\psi}-\tilde{v}}^2_{L^2}\left\lVert\int_{-\infty}^x|\varphi|^{2(\sigma-2)}\Im((\tilde{\psi}+k)^2\overline{\varphi}^2)\,dy\right\rVert_{L^{\infty}_x}\nonumber\\
&\quad+\norm{\tilde{\psi}-\tilde{v}}_{L^2}\norm{\tilde{v}+k}_{L^2}\left\lVert\int_{-\infty}^x|\varphi|^{2(\sigma-2)}\Im(\overline{\varphi}^2((\tilde{\psi}+k)^2-(\tilde{v}+k)^2))\,dy\right\rVert_{L^{\infty}_x}\nonumber\\
&\lesssim  \norm{\tilde{\psi}-\tilde{v}}^2_{L^2}\norm{\varphi}^{2\sigma-1}_{L^{\infty}}\norm{\tilde{\psi}+\tilde{v}+2k}_{L^{\infty}}+\norm{\tilde{\psi}-\tilde{v}}^2_{L^2}\norm{\varphi^{2(\sigma-1)}(\tilde{\psi}+k)^2}_{L^1_x}\nonumber\\
&\quad +\norm{\tilde{\psi}-\tilde{v}}_{L^2}\norm{\tilde{v}+k}_{L^2}\norm{\varphi^{2(\sigma-1)}((\tilde{\psi}+k)^2-(\tilde{v}+k)^2)}_{L^1}\nonumber\\
&\lesssim \norm{\tilde{\psi}-\tilde{v}}^2_{L^2}\norm{\varphi}^{2\sigma-1}_{L^{\infty}}\norm{\tilde{\psi}+\tilde{v}+2k}_{L^{\infty}}+\norm{\tilde{\psi}-\tilde{v}}^2_{L^2}\norm{\varphi^{2(\sigma-1)}(\tilde{\psi}+k)^2}_{L^1}\nonumber\\
&\quad+\norm{\tilde{\psi}-\tilde{v}}^2_{L^2}\norm{\tilde{v}+k}_{L^2}\norm{\varphi^{2(\sigma-1)}(\tilde{\psi}+\tilde{v}+2k)}_{L^2}\nonumber\\
&\lesssim \norm{\tilde{\psi}-\tilde{v}}^2_{L^2}K_1, \label{estimate of A}   
\end{align}
where,
\begin{align*}
K_1&:=\norm{\varphi}^{2\sigma-1}_{L^{\infty}}\norm{\tilde{\psi}+\tilde{v}+2k}_{L^{\infty}}+\norm{\varphi^{2(\sigma-1)}(\tilde{\psi}+k)^2}_{L^1}+\norm{\tilde{v}+k}_{L^2}\norm{\varphi^{2(\sigma-1)}(\tilde{\psi}+\tilde{v}+2k)}_{L^2}.
\end{align*}
Furthermore, 
\begin{align}
|B|&\lesssim \left|\int_{\R}\partial_x(|\varphi|^{2(\sigma-1)}\varphi^2(\overline{\tilde{\psi}}-\overline{\tilde{v}}))(\overline{\tilde{\psi}}-\overline{\tilde{v}})\,dx\right|\nonumber\\
&\quad+\left|\int_{\R}\partial_x\left(\varphi\int_{-\infty}^x|\varphi|^{2(\sigma-2)}\Im(\overline{\varphi}^2((\tilde{\psi}+k)^2-(\tilde{v}+k)^2))\,dy\right)(\overline{\tilde{\psi}}-\overline{\tilde{v}})\,dx\right|\nonumber\\
&\lesssim \left|\int_{\R}\partial_x(|\varphi|^{2(\sigma-1)}\varphi^2)(\overline{\tilde{\psi}}-\overline{\tilde{v}})^2\,dx\right|+\left||\varphi|^{2(\sigma-1)}\varphi^2\frac{1}{2}\partial_x((\tilde{\psi}-\tilde{v})^2)\,dx\right|\label{eqqa1}\\
& \quad+\left|\int_{\R}\partial_x\varphi\int_{-\infty}^x|\varphi|^{2(\sigma-2)}\Im(\overline{\varphi}^2(\tilde{\psi}-\tilde{v})(\tilde{\psi}+\tilde{v}+2k))\,dy(\overline{\tilde{\psi}}-\overline{\tilde{v}})\,dx\right|\nonumber\\
&\quad+\left|\int_{\R}\varphi|\varphi|^{2(\sigma-2)}\Im(\overline{\varphi}^2(\tilde{\psi}-\tilde{v})(\tilde{\psi}+\tilde{v}+2k))(\overline{\tilde{\psi}}-\overline{\tilde{v}})\,dx\right|.\nonumber
\end{align}
By using integration by parts for the second term of \eqref{eqqa1} and using H\"older inequality we have
\begin{align}
|B|&\lesssim \norm{\tilde{\psi}-\tilde{v}}_{L^2}^2\norm{\partial_x(|\varphi|^{2(\sigma-1)}\varphi^2)}_{L^{\infty}}+\norm{\tilde{\psi}-\tilde{v}}_{L^2}^2\norm{\partial_x(|\varphi|^{2(\sigma-1)}\varphi^2)}_{L^{\infty}} \nonumber\\
&\quad+\norm{\partial_x\varphi}_{L^2}\norm{\int_{-\infty}^x|\varphi|^{2(\sigma-2)}\Im(\overline{\varphi}^2(\tilde{\psi}-\tilde{v})(\tilde{\psi}+\tilde{v}+2k))\,dy}_{L^{\infty}_x}\norm{\tilde{\psi}-\tilde{v}}_{L^2}\nonumber\\
&\quad +\norm{\tilde{\psi}-\tilde{v}}_{L^2}^2\norm{\varphi^{2\sigma-1}(\tilde{\psi}+\tilde{v}+2k)}_{L^{\infty}}\nonumber\\
&\lesssim \norm{\tilde{\psi}-\tilde{v}}_{L^2}^2\norm{\partial_x(|\varphi|^{2(\sigma-1)}\varphi^2)}_{L^{\infty}}+\norm{\tilde{\psi}-\tilde{v}}_{L^2}^2\norm{\partial_x(|\varphi|^{2(\sigma-1)}\varphi^2)}_{L^{\infty}} \nonumber\\
&\quad+\norm{\partial_x\varphi}_{L^2}\norm{\tilde{\psi}-\tilde{v}}_{L^2}\norm{\varphi^{2(\sigma-1)}(\tilde{\psi}-\tilde{v})(\tilde{\psi}+\tilde{v}+2k)}_{L^1_x}\\
&\quad+\norm{\tilde{\psi}-\tilde{v}}_{L^2}^2\norm{\varphi^{2\sigma-1}(\tilde{\psi}+\tilde{v}+2k)}_{L^{\infty}}\nonumber\\
&\lesssim \norm{\tilde{\psi}-\tilde{v}}_{L^2}^2\norm{\partial_x(|\varphi|^{2(\sigma-1)}\varphi^2)}_{L^{\infty}}+\norm{\tilde{\psi}-\tilde{v}}_{L^2}^2\norm{\partial_x(|\varphi|^{2(\sigma-1)}\varphi^2)}_{L^{\infty}} \nonumber\\
&\quad+\norm{\partial_x\varphi}_{L^2}\norm{\tilde{\psi}-\tilde{v}}^2_{L^2}\norm{\varphi^{2(\sigma-1)}(\tilde{\psi}+\tilde{v}+2k)}_{L^2}+\norm{\tilde{\psi}-\tilde{v}}_{L^2}^2\norm{\varphi^{2\sigma-1}(\tilde{\psi}+\tilde{v}+2k)}_{L^{\infty}}\nonumber\\
&= \norm{\tilde{\psi}-\tilde{v}}_{L^2}^2K_2,\label{estimate for B}
\end{align}
where
\[
K_2:=\norm{\partial_x(|\varphi|^{2(\sigma-1)}\varphi^2)}_{L^{\infty}}+\norm{\partial_x\varphi}_{L^2}\norm{\varphi^{2(\sigma-1)}(\tilde{\psi}+\tilde{v}+2k)}_{L^2}+\norm{\varphi^{2\sigma-1}(\tilde{\psi}+\tilde{v}+2k)}_{L^{\infty}}.
\]
Using \eqref{eq of P}, we have
\begin{align}
|C|&\lesssim \left|\int_{\R}|\varphi|^{2\sigma}|\varphi|^{2(\sigma-1)}\varphi^2(\overline{\tilde\psi}-\overline{\tilde{v}})^2\,dx\right|\nonumber\\
&\quad+\left|\int_{\R}|\varphi|^{2\sigma}\varphi\int_{-\infty}^x|\varphi|^{2(\sigma-2)}\Im(\overline{\varphi}^2((\tilde{\psi}+k)^2-(\tilde{v}+k)^2))\,dy(\overline{\tilde{\psi}}-\overline{\tilde{v}})\,dx\right|\nonumber\\
&\lesssim \norm{\tilde{\psi}-\tilde{v}}^2_{L^2}\norm{\varphi^{4\sigma}}_{L^{\infty}}\nonumber\\
&\quad+\norm{\tilde{\psi}-\tilde{v}}_{L^2}\norm{\varphi^{2\sigma+1}}_{L^2}\norm{\int_{-\infty}^x|\varphi|^{2(\sigma-2)}\Im(\overline{\varphi}^2(\tilde{\psi}-\tilde{v})(\tilde{\psi}+\tilde{v}+2k))\,dy}_{L^{\infty}_x}\nonumber\\
&\lesssim \norm{\tilde{\psi}-\tilde{v}}^2_{L^2}\norm{\varphi^{4\sigma}}_{L^{\infty}}+\norm{\tilde{\psi}-\tilde{v}}_{L^2}\norm{\varphi^{2\sigma+1}}_{L^2}\norm{\varphi^{2(\sigma-1)}(\tilde{\psi}-\tilde{v})(\tilde{\psi}+\tilde{v}+2k)}_{L^1}\nonumber\\
&\lesssim \norm{\tilde{\psi}-\tilde{v}}^2_{L^2}\norm{\varphi^{4\sigma}}_{L^{\infty}}+\norm{\tilde{\psi}-\tilde{v}}^2_{L^2}\norm{\varphi^{2\sigma+1}}_{L^2}\norm{\varphi^{2(\sigma-1)}(\tilde{\psi}+\tilde{v}+2k)}_{L^2}\nonumber\\
&=\norm{\tilde{\psi}-\tilde{v}}^2_{L^2}K_3,\label{estimate for C}
\end{align}
where
\[
K_3:=\norm{\varphi^{4\sigma}}_{L^{\infty}}+\norm{\varphi^{2\sigma+1}}_{L^2}\norm{\varphi^{2(\sigma-1)}(\tilde{\psi}+\tilde{v}+2k)}_{L^2}.
\] 
Now, we give an estimate for $D$. We have
\begin{align}
|D|&\lesssim \left|\int_{\R}|\varphi|^{2(\sigma-1)}\varphi^2|\varphi|^{2(\sigma-1)}\overline{\varphi}^2(\tilde{\psi}-\tilde{v})(\overline{\tilde{\psi}}-\overline{\tilde{v}})\,dx\right|\nonumber\\
&\quad+\left|\int_{\R}|\varphi|^{2(\sigma-1)}\varphi^2\varphi\int_{-\infty}^x|\varphi|^{2(\sigma-2)}\Im(\overline{\varphi}^2((\tilde{\psi}+k)^2-(\tilde{v}+k)^2))\,dy(\overline{\tilde{\psi}}-\overline{\tilde{v}})\,dx\right|\nonumber\\
&\lesssim \norm{\tilde{\psi}-\tilde{v}}^2_{L^2}\norm{\varphi^{4\sigma}}_{L^{\infty}}\nonumber\\
&\quad+\norm{\tilde{\psi}-\tilde{v}}_{L^2}\norm{\varphi^{2\sigma+1}}_{L^2}\norm{\int_{-\infty}^x|\varphi|^{2(\sigma-2)}\Im(\overline{\varphi}^2((\tilde{\psi}+k)^2-(\tilde{v}+k)^2))\,dy}_{L^{\infty}_x}\nonumber\\
&\lesssim \norm{\tilde{\psi}-\tilde{v}}^2_{L^2}\norm{\varphi^{4\sigma}}_{L^{\infty}}+\norm{\tilde{\psi}-\tilde{v}}_{L^2}\norm{\varphi^{2\sigma+1}}_{L^2}\norm{\varphi^{2(\sigma-1)}(\tilde{\psi}-\tilde{v})(\tilde{\psi}+\tilde{v}+2k)}_{L^1}\nonumber\\
&\lesssim \norm{\tilde{\psi}-\tilde{v}}^2_{L^2}\norm{\varphi^{4\sigma}}_{L^{\infty}}+\norm{\tilde{\psi}-\tilde{v}}^2_{L^2}\norm{\varphi^{2\sigma+1}}_{L^2}\norm{\varphi^{2(\sigma-1)}(\tilde{\psi}+\tilde{v}+2k)}_{L^2}\nonumber\\
&=\norm{\tilde{\psi}-\tilde{v}}^2_{L^2}K_4,\label{estimate for D}
\end{align}
where
\[
K_4:=\norm{\varphi^{4\sigma}}_{L^{\infty}}+\norm{\varphi^{2\sigma+1}}_{L^2}\norm{\varphi^{2(\sigma-1)}(\tilde{\psi}+\tilde{v}+2k)}_{L^2}.
\]
Combining \eqref{estimate of A}, \eqref{estimate for B}, \eqref{estimate for C} and \eqref{estimate for D}, we have
\begin{align*}
\left|\partial_t\norm{\tilde{\psi}-\tilde{v}}^2_{L^2}\right|&\lesssim \norm{\tilde{\psi}-\tilde{v}}^2_{L^2}(K_1+K_2+K_3+K_4).
\end{align*}
Using the Gr\"onwall inequality, we have
\begin{align}
\norm{\tilde{\psi}(t)-\tilde{v}(t)}^2_{L^2}&\lesssim \norm{\tilde{\psi}(N)-\tilde{v}(N)}^2_{L^2}\exp\left(\int_t^N (K_1+K_2+K_3+K_4)\,ds\right) \nonumber\\
&\leq e^{-2\lambda N}\exp\left(\int_t^N (K_1+K_2+K_3+K_4)\,ds\right).\label{estimate nice}
\end{align}
Now, we try to estimate $K_1+K_2+K_3+K_4$ in term of $R$. When we have this kind of estimate, we will use the assumption \eqref{condition of existence multi soliton} to obtain that $\tilde{\psi}=\tilde{v}$. We have
\begin{align}
&\int_t^N(K_1+K_2+K_3+K_4)\,ds\nonumber\\
&=\int_t^N \norm{\varphi}^{2\sigma-1}_{L^{\infty}}\norm{\tilde{\psi}+\tilde{v}+2k}_{L^{\infty}}+\norm{\varphi^{2(\sigma-1)}(\tilde{\psi}+k)^2}_{L^1}\nonumber\\
&\quad+\norm{\tilde{v}+k}_{L^2}\norm{\varphi^{2(\sigma-1)}(\tilde{\psi}+\tilde{v}+2k)}_{L^2}\,ds\label{term1new}\\
&\quad+\int_t^N\norm{\partial_x(|\varphi|^{2(\sigma-1)}\varphi^2)}_{L^{\infty}}+\norm{\partial_x\varphi}_{L^2}\norm{\varphi^{2(\sigma-1)}(\tilde{\psi}+\tilde{v}+2k)}_{L^2}\nonumber\\
&\quad+\norm{\varphi^{2\sigma-1}(\tilde{\psi}+\tilde{v}+2k)}_{L^{\infty}}\,ds\label{term2new}\\
&\quad+\int_t^N\norm{\varphi^{4\sigma}}_{L^{\infty}}+\norm{\varphi^{2\sigma+1}}_{L^2}\norm{\varphi^{2(\sigma-1)}(\tilde{\psi}+\tilde{v}+2k)}_{L^2}\,ds\label{term3new}\\
&\quad+\int_t^N\norm{\varphi^{4\sigma}}_{L^{\infty}}+\norm{\varphi^{2\sigma+1}}_{L^2}\norm{\varphi^{2(\sigma-1)}(\tilde{\psi}+\tilde{v}+2k)}_{L^2}\,ds\label{term4new} 
\end{align}
Using \eqref{estimate on solution of system} and \eqref{estimate on H1 of solution of system}, we have
\begin{align}
\norm{\varphi}_{L^{\infty}}&\leq\norm{\tilde{\varphi}}_{L^{\infty}}+\norm{h}_{L^{\infty}}\lesssim 1+\norm{h}_{L^{\infty}}\label{equseful1}\\
\norm{\varphi}_{L^2}&\leq \norm{\tilde{\varphi}}_{L^2}+\norm{h}_{L^2}\lesssim 1+\norm{h}_{L^2}\label{equseful2}\\
\norm{\psi}_{L^{\infty}}&\lesssim 1\label{equseful3}
\end{align}
We denote by $Z_1,Z_2,Z_3,Z_4$ the terms \eqref{term1new}, \eqref{term2new}, \eqref{term3new} and \eqref{term4new} respectively. Using \eqref{equseful1}, \eqref{equseful2}, \eqref{equseful3}, \eqref{estimate on solution of system} and \eqref{estimate on H1 of solution of system}, for $N \gg t$, we have
\begin{align*}
|Z_1|&\lesssim \norm{\varphi}^3_{L^4(t,N)L^{\infty}}\norm{\varphi}^{2(\sigma-2)}_{L^{\infty}L^{\infty}}\norm{\tilde{\psi}+\tilde{v}+2k}_{L^4(t,N)L^{\infty}}\\
&\quad+(N-t)\norm{\varphi}^{2(\sigma-1)}_{L^{\infty}L^{\infty}}(\norm{\tilde{\psi}}_{L^{\infty}L^2}+\norm{k}_{L^{\infty}L^2})^2\\
&\quad+\norm{\tilde{v}+k}_{L^{\frac{4}{3}}(t,N)L^2}\norm{\varphi}_{L^{\infty}L^2}\norm{\varphi}^{2(\sigma-1)}_{L^{\infty}L^{\infty}}(\norm{\tilde{\psi}+\tilde{v}}_{L^4(t,N)L^{\infty}}+\norm{k}_{L^4(t,N)L^{\infty}})\\
&\lesssim (N-t)^{\frac{3}{4}}\norm{\varphi}^{2\sigma-1}_{L^{\infty}L^{\infty}}(1+\norm{k}_{L^{\infty}L^{\infty}}(N-t)^{\frac{1}{4}})\\
&\quad+(N-t)(1+\norm{h}^{2(\sigma-1)}_{L^{\infty}L^{\infty}})(1+\norm{k}^2_{L^{\infty}L^2})\\
&\quad+(N-t)^{\frac{3}{4}}(1+\norm{k}_{L^{\infty}L^2})(1+\norm{h}_{L^{\infty}L^2})(1+\norm{h}^{2(\sigma-1)}_{L^{\infty}L^{\infty}})(1+(N-t)^{\frac{1}{4}}\norm{k}_{L^{\infty}L^{\infty}})\\
&\lesssim (N-t)\norm{k}_{L^{\infty}L^{\infty}}(1+\norm{h}^{2\sigma-1}_{L^{\infty}L^{\infty}})+(N-t)(1+\norm{h}^{2(\sigma-1)}_{L^{\infty}L^{\infty}})(1+\norm{k}^2_{L^{\infty}L^2})\\
&\quad+(N-t)\norm{k}_{L^{\infty}L^{\infty}}(1+\norm{k}_{L^{\infty}L^2})(1+\norm{h}_{L^{\infty}L^2})(1+\norm{h}^{2(\sigma-1)}_{L^{\infty}L^{\infty}})\\
&:=(N-t)W_1(h,k).
\end{align*}
Similarly, for $N\gg t$, we have
\begin{align*}
|Z_2|&\lesssim \norm{\partial_x\varphi\varphi^{2\sigma-1}}_{L^1(t,N)L^{\infty}}+(N-t)\norm{\partial_x\varphi}_{L^{\infty}(t,N)L^2}\norm{\varphi}^{2(\sigma-1)}_{L^{\infty}L^{\infty}}\norm{\tilde{\psi}+\tilde{v}+k}_{L^{\infty}(t,N)L^2}\\
&\quad+(N-t)^{\frac{3}{4}}\norm{\varphi}^{2\sigma-1}_{L^{\infty}L^{\infty}}(\norm{\tilde{\psi}+\tilde{v}}_{L^4(t,N)L^{\infty}}+\norm{k}_{L^4(t,N)L^{\infty}})\\
&\lesssim (N-t)^{\frac{3}{4}}(\norm{\partial_x\tilde\varphi}_{L^4(t,N)L^{\infty}}+\norm{\partial_xh}_{L^4(t,N)L^{\infty}})\norm{\varphi}^{2\sigma-1}_{L^{\infty}L^{\infty}}\\
&\quad+(N-t)(1+\norm{h}^{2(\sigma-1)}_{L^{\infty}L^{\infty}})(1+\norm{k}_{L^{\infty}L^2})\\
&\quad+(N-t)^{\frac{3}{4}}(1+\norm{h}^{2\sigma-1}_{L^{\infty}L^{\infty}})(1+(N-t)^{\frac{1}{4}}\norm{k}_{L^{\infty}L^{\infty}})\\
&\lesssim (N-t)\norm{\partial_xh}_{L^{\infty}L^{\infty}}(1+\norm{h}^{2\sigma-1}_{L^{\infty}L^{\infty}})+ (N-t)(1+\norm{h}^{2(\sigma-1)}_{L^{\infty}L^{\infty}})(1+\norm{k}_{L^{\infty}L^2})\\
&\quad+(N-t)\norm{k}_{L^{\infty}L^{\infty}}(1+\norm{h}^{2\sigma-1}_{L^{\infty}L^{\infty}})\\
&:=(N-t)W_2(h,k),
\end{align*}
and
\begin{align*}
&|Z_3|=|Z_4|\\
&\lesssim (N-t)(\norm{\tilde{\varphi}}_{L^{\infty}L^{\infty}}+\norm{h}_{L^{\infty}L^{\infty}})^{4\sigma}\\
&\quad+(N-t)\norm{\varphi}_{L^{\infty}L^2}\norm{\varphi}^{2\sigma}_{L^{\infty}L^{\infty}}\norm{\varphi}^{2(\sigma-1)}_{L^{\infty}L^{\infty}}(\norm{\tilde{\psi}+\tilde{v}}_{L^{\infty}L^2}+\norm{k}_{L^{\infty}L^2}) \\
&\lesssim (N-t)(1+\norm{h}^{4\sigma}_{L^{\infty}L^{\infty}})+(N-t)(1+\norm{h}_{L^{\infty}L^2})(1+\norm{h}^{4\sigma-2}_{L^{\infty}L^{\infty}})(1+\norm{k}_{L^{\infty}L^2})\\
&:=(N-t)W_3(h,k).
\end{align*}
Hence, from \eqref{estimate nice}, we have
\begin{align}
\norm{\tilde{\psi(t)}-\tilde{v(t)}}_{L^2}^2&\lesssim e^{-2\lambda N}\exp\left(\int_t^N(K_1+K_2+K_3+K_4)\,ds\right)\nonumber\\
&\lesssim e^{-2\lambda N}\exp((N-t)(W_1(h,k)+W_2(h,k)+W_3(h,k)))\label{estimate nice new}
\end{align}
The above estimate is not enough explicit. As said above, we would like to estimate the right hand side of \eqref{estimate nice new} in terms of $R$. Noting that $|h|=|R|$ and $|k|=|\partial_xR|$, we have
\begin{align*}
W_1(h,k)&=\norm{\partial_xR}_{L^{\infty}L^{\infty}}(1+\norm{R}^{2\sigma-1}_{L^{\infty}L^{\infty}})+(1+\norm{R}^{2(\sigma-1)}_{L^{\infty}L^{\infty}})(1+\norm{\partial_xR}^2_{L^{\infty}L^2})\\
&\quad+\norm{\partial_xR}_{L^{\infty}L^{\infty}}(1+\norm{\partial_xR}_{L^{\infty}L^2})(1+\norm{R}_{L^{\infty}L^2})(1+\norm{R}^{2(\sigma-1)}_{L^{\infty}L^{\infty}})  \\
&\lesssim (1+\norm{R}_{L^{\infty}L^{\infty}}^{2(\sigma-1)})\left[\norm{\partial_xR}_{L^{\infty}L^{\infty}}(1+\norm{R}_{L^{\infty}L^{\infty}})+(1+\norm{\partial_xR}_{L^{\infty}L^2})\right.\\
&\left.+\norm{\partial_xR}_{L^{\infty}L^{\infty}}(1+\norm{\partial_xR}_{L^{\infty}L^2})(1+\norm{R}_{L^{\infty}L^2})\right]\\
&\lesssim (1+\norm{R}_{L^{\infty}L^{\infty}}^{2(\sigma-1)})\times \\
&\quad \times \left[\norm{\partial_xR}_{L^{\infty}L^{\infty}}(1+\norm{R}_{L^{\infty}H^1})+(1+\norm{R}_{L^{\infty}H^1}^2)+\norm{\partial_xR}_{L^{\infty}L^{\infty}}(1+\norm{R}_{L^{\infty}H^1}^2)\right]\\
&\lesssim (1+\norm{R}_{L^{\infty}L^{\infty}}^{2(\sigma-1)})(1+\norm{R}_{L^{\infty}H^1}^2)(1+\norm{\partial_xR}_{L^{\infty}L^{\infty}}).
\end{align*}
Similarly, by noting that $|\partial_xh|\leq |k|+|h|^{2\sigma+1}$, we have
\begin{align*}
W_2(h,k)&\lesssim (\norm{k}_{L^{\infty}L^{\infty}}+\norm{h}^{2\sigma+1}_{L^{\infty}L^{\infty}})(1+\norm{h}^{2(\sigma-1)}_{L^{\infty}L^{\infty}})(1+\norm{h}_{L^{\infty}L^{\infty}})\\
&\quad+(1+\norm{h}^{2(\sigma-1)})(1+\norm{k}_{L^{\infty}L^2})+\norm{k}_{L^{\infty}L^{\infty}}(1+\norm{h}^{2(\sigma-1)}_{L^{\infty}L^{\infty}})(1+\norm{h}_{L^{\infty}L^{\infty}})\\
&\lesssim (1+\norm{h}^{2(\sigma-1)})\times\\
&\quad\times\left[(\norm{k}_{L^{\infty}L^{\infty}}+\norm{h}^{2\sigma+1}_{L^{\infty}L^{\infty}})(1+\norm{h}_{L^{\infty}L^{\infty}})\right.\nonumber\\
&\quad\left.+(1+\norm{k}_{L^{\infty}L^2})+\norm{k}_{L^{\infty}L^{\infty}}(1+\norm{h}_{L^{\infty}L^{\infty}})\right]\\
&\lesssim (1+\norm{h}^{2(\sigma-1)})\times\\
&\quad\times\left[(1+\norm{h}_{L^{\infty}L^{\infty}})(\norm{k}_{L^{\infty}L^{\infty}}+\norm{h}^{2\sigma+1}_{L^{\infty}L^{\infty}})+(1+\norm{k}_{L^{\infty}L^2})\right]\\
&=(1+\norm{R}_{L^{\infty}L^{\infty}}^{2(\sigma-1)})\times\\
&\quad\times\left[(1+\norm{R}_{L^{\infty}L^{\infty}})(\norm{\partial_xR}_{L^{\infty}L^{\infty}}+\norm{R}_{L^{\infty}L^{\infty}}^{2\sigma+1})+(1+\norm{\partial_xR}_{L^{\infty}L^2})\right]\\
&\lesssim (1+\norm{R}_{L^{\infty}L^{\infty}}^{2(\sigma-1)})(1+\norm{R}_{L^{\infty}H^1})(1+\norm{\partial_xR}_{L^{\infty}L^{\infty}}+\norm{R}_{L^{\infty}L^{\infty}}^{2\sigma+1})\\
&\lesssim (1+\norm{R}_{L^{\infty}L^{\infty}}^{2(\sigma-1)})(1+\norm{R}_{L^{\infty}H^1}^2)(1+\norm{\partial_xR}_{L^{\infty}L^{\infty}}+\norm{R}_{L^{\infty}L^{\infty}}^{2\sigma+1}),   
\end{align*}
and 
\begin{align*}
W_3(h,k)&=(1+\norm{R}_{L^{\infty}L^{\infty}}^{4\sigma})+(1+\norm{R}_{L^{\infty}L^2})(1+\norm{R}_{L^{\infty}L^{\infty}}^{4\sigma-2})(1+\norm{\partial_xR}_{L^{\infty}L^2})\\
&\lesssim (1+\norm{R}_{L^{\infty}L^{\infty}}^{4\sigma-2})\left[(1+\norm{R}_{L^{\infty}L^{\infty}}^2)+(1+\norm{R}_{L^{\infty}L^2})(1+\norm{\partial_xR}_{L^{\infty}L^2})\right]\\
&\lesssim (1+\norm{R}_{L^{\infty}L^{\infty}}^{4\sigma-2})(1+\norm{R}_{L^{\infty}H^1}^2).
\end{align*}
Combining the above estimates, we have
\begin{align*}
&W_1(h,k)+W_2(h,k)+W_3(h,k)\\
&\lesssim (1+\norm{R}_{L^{\infty}L^{\infty}}^{2(\sigma-1)})(1+\norm{R}_{L^{\infty}H^1}^2)(1+\norm{\partial_xR}_{L^{\infty}L^{\infty}}+\norm{R}_{L^{\infty}L^{\infty}}^{2\sigma+1})\\
&\quad+(1+\norm{R}_{L^{\infty}L^{\infty}}^{4\sigma-2})(1+\norm{R}_{L^{\infty}H^1}^2)\\
&\lesssim (1+\norm{R}_{L^{\infty}L^{\infty}}^{2(\sigma-1)})(1+\norm{R}_{L^{\infty}H^1}^2)(1+\norm{\partial_xR}_{L^{\infty}L^{\infty}}+\norm{R}_{L^{\infty}L^{\infty}}^{2\sigma+1})\\
&\quad+(1+\norm{R}_{L^{\infty}L^{\infty}}^{2(\sigma-1)})(1+\norm{R}_{L^{\infty}L^{\infty}}^{2\sigma})(1+\norm{R}_{L^{\infty}H^1}^2)\\
&\lesssim (1+\norm{R}_{L^{\infty}L^{\infty}}^{2(\sigma-1)})(1+\norm{R}_{L^{\infty}H^1}^2)(1+\norm{\partial_xR}_{L^{\infty}L^{\infty}}+\norm{R}_{L^{\infty}L^{\infty}}^{2\sigma+1}).
\end{align*}
Thus, there exists a positive constant $C_0$ such that 
\begin{align*}
&W_1(h,k)+W_2(h,k)+W_3(h,k)\\
&\leq C_0\left((1+\norm{R}_{L^{\infty}L^{\infty}}^{2(\sigma-1)})(1+\norm{R}_{L^{\infty}H^1}^2)(1+\norm{\partial_xR}_{L^{\infty}L^{\infty}}+\norm{R}_{L^{\infty}L^{\infty}}^{2\sigma+1})\right).
\end{align*}
Let $C_*=16C_0$. Using the assumption \eqref{condition of existence multi soliton}, we have
\[
W_1(h,k)+W_2(h,k)+W_3(h,k) \leq \frac{v_*}{16}=\lambda,
\] 
for $t$ large enough. Thus, by \eqref{estimate nice new}, we have
\[
\norm{\tilde{\psi}(t)-\tilde{v}(t)}^2_{L^2}\leq e^{-2\lambda N+(N-t)\lambda},
\]
for $t$ large enough. Letting $N \rightarrow \infty$ in the above estimate, we obtain 
\[
\norm{\tilde{\psi}(t)-\tilde{v}}_{L^2}^2=0,
\]
for all $t$ large enough. This implies that 
\begin{equation}\label{relation of psi and varphi new}
\tilde{\psi}=\partial_x\varphi-\frac{i}{2}|\varphi|^2\varphi-k,
\end{equation}
and then 
\[
\psi=\partial_x\varphi-\frac{i}{2}|\varphi|^2\varphi. 
\]
Moreover, since $(\tilde{\psi},\tilde{\varphi})$ solves \eqref{system new} we have $(\psi,\varphi)$ solves \eqref{system equation}. Combining with \eqref{relation of psi and varphi new}, if we set 
\[
u=\exp\left(-\frac{i}{2}\int_{-\infty}^x|\varphi|^{2\sigma}\,dy\right)\varphi
\]
then $u$ solves \eqref{gdnls}. Furthermore,
\begin{align*}
\norm{u-R}_{H^1} &=\left\lVert\exp\left(-\frac{i}{2}|\varphi|^{2\sigma}\,dy\right)\varphi-\exp\left(\frac{i}{2}|h|^{2\sigma}\,dy\right)h\right\rVert_{H^1}\\
&\lesssim C(\norm{\varphi}_{H^1},\norm{h}_{H^1}) \norm{\varphi-h}_{H^1}\lesssim \norm{\tilde{\varphi}}_{H^1}\lesssim e^{-\lambda t},
\end{align*} 
Thus for $t$ large enough, we have
\begin{equation}\label{eq finally}
\norm{u-R}_{H^1} \leq Ce^{-\lambda t},
\end{equation}
for $\lambda=\frac{1}{16}v_{*}$ and $C=C(\omega_1,...,\omega_K,c_1,...,c_K)$. This completes the proof of Theorem \ref{main result}.

\begin{remark}
In the case $\sigma=1$, the integrals in \eqref{system equation} disappear. In the case, $\sigma=2$, the integrals \eqref{system equation} reduce into $\int_{-\infty}^x\Im(\psi^2\overline{\varphi}^2)\,dy$, we do not need to use the inequality \eqref{eq 11111}. Thus, by similar arguments as in the proof of Theorem \ref{main result} we may prove that there exist multi-solitons solutions of \eqref{gdnls} when $\sigma=1$ or $\sigma=2$.
\end{remark}

\section{Some technical lemmas}
\label{section52}
\subsection{Properties of solitons}
In this section, we give the proof of \eqref{estimate need to prove}. We have the following result. 
\begin{lemma}\label{estimate on profile}
There exist $C>0$ and a constant $\lambda>0$ such that for $t>0$ large enough, the estimate \eqref{estimate need to prove} uniformly holds in time.
\end{lemma}

\begin{proof}
First, we need some estimates on the profile. We have
\begin{align*}
|R_j(t,x)|&=|\psi_{\omega_j,c_j}(t,x)|=|\phi_{\omega_j,c_j}(x-c_jt)|=|\varphi_{\omega_j,c_j}(x-c_jt)|\\
&\quad\approx\left(\frac{4\omega_j-c_j^2}{2\sqrt{\omega_j}\left(\cosh(\sigma h_j (x-c_jt))-\frac{c_j}{2\sqrt{\omega_j}}\right)}\right)^{\frac{1}{2\sigma}}\\
&\lesssim \left(\frac{4\omega_j-c_j^2}{2\sqrt{\omega_j}\left(\cosh(\sigma h_j(x-c_jt))-\frac{|c_j|}{2\sqrt{\omega_j}}\cosh(\sigma h_j(x-c_jt))\right)}\right)^{\frac{1}{2\sigma}}\\
&\lesssim \left(\frac{4\omega_j-c_j^2}{(2\sqrt{\omega_j}-|c_j|)\cosh(\sigma h_j(x-c_jt))}\right)^{\frac{1}{2\sigma}}\lesssim \left(\frac{2\sqrt{\omega_j}+|c_j|}{\cosh(\sigma h_j (x-c_jt))}\right)^{\frac{1}{2\sigma}}\\
&\lesssim_{\omega_j,|c_j|} e^{-\frac{h_j}{2}|x-c_jt|}, 
\end{align*}
Furthermore,
\begin{align*}
\partial_x\varphi_{\omega_j,c_j}(y)&\approx \left(\frac{h_j^2}{2\sqrt{\omega_j}}\right)^{\frac{1}{2\sigma}}\frac{-\sinh(\sigma h_jy)}{\left(\cosh(\sigma h_j y)-\frac{c_j}{\sqrt{\omega_j}}\right)^{1+\frac{1}{2\sigma}}}.
\end{align*}
Thus,
\begin{align*}
|\partial_x\varphi_{\omega_j,c_j}(y)|&\lesssim \left(\frac{h_j^2}{2\sqrt{\omega_j}}\right)^{\frac{1}{2\sigma}}\frac{|\sinh(\sigma h_j y)|}{\left(1-\frac{|c_j|}{\sqrt{\omega_j}}\right)^{1+\frac{1}{2\sigma}} \cosh(\sigma h_j y)^{1+\frac{1}{2\sigma}}}\\
&\lesssim_{\omega_j,|c_j|}\frac{1}{\cosh(\sigma h_jy)^{\frac{1}{2\sigma}}} \lesssim_{\omega_j,|c_j|} e^{-\frac{h_j}{2}|y|},
\end{align*}
Using the above estimates, we have
\begin{align*}
|\partial_xR_j(t,x)|&=|\partial_x\psi_{\omega_j,c_j}(t,x)|=|\partial_x\phi_{\omega_j,c_j}(x-c_jt)|\\
&=|\partial_x\varphi_{\omega_j,c_j}(x-c_jt)+i\varphi_{\omega_j,c_j}(x-c_jt)\partial_x\theta_{\omega_j,c_j}(x-c_jt)|\\
&\lesssim |\partial_x\varphi_{\omega_j,c_j}(x-c_jt)|+|\varphi_{\omega_j,c_j}(x-c_jt)||\partial_x\theta_{\omega_j,c_j}(x-c_jt)|\\
&\lesssim_{\omega_j,|c_j|} |\partial_x\varphi_{\omega_j,c_j}(x-c_jt)|+e^{\frac{-h_j}{2}|x-c_jt|}\\
&\lesssim_{\omega_j,|c_j|}  e^{-\frac{h_j}{2}|x-c_jt|}.
\end{align*}
By similar arguments, we have
\begin{align*}
|\partial_x^2R_j(t,x)|+|\partial_x^3R_j(t,x)| &\lesssim_{\omega_j,|c_j|}e^{\frac{-h_j}{2}|x-c_jt|},
\end{align*}
For convenience, we set
\begin{align*}
\chi&=-i|R|^{2\sigma}\partial_xR+i\Sigma_{j}|R_j|^{2\sigma}\partial_xR_{j},\\
f(R,\overline{R},\partial_xR)&= i|R|^{2\sigma}\partial_xR,\\
g(R,\overline{R},\partial_xR,\partial_x\overline{R},\partial_x^2R)&=i\partial_x(|R|^{2\sigma}\partial_xR),\\
r(R,\partial_xR,..,\partial_x^3R,\partial_x\overline{R},\partial_x^2\overline{R})&=i\partial_x^2(|R|^{2\sigma}\partial_xR).
\end{align*}
Fix $t>0$, for each $x\in\R$, choose $m=m(x)\in \left\{1,2,...,K\right\}$ so that
\[
|x-c_mt|=\min_{j}|x-c_jt|.
\] 
For $j \neq m$ we have
\[
|x-c_jt|\geq \frac{1}{2}(|x-c_jt|+|x-c_mt|)\geq \frac{1}{2}|c_jt-c_mt|=\frac{t}{2}|c_j-c_m|.
\]
Thus, we have
\begin{align*}
&|(R-R_m)(t,x)|+|\partial_x(R-R_m)(t,x)|+|\partial_x^2(R-R_m)(t,x)|+|\partial_x^3(R-R_m)(t,x)|\\
&\leq \sum_{j\neq m}(|R_j(t,x)|+|\partial_xR_j(t,x)|+|\partial_x^2R_j(t,x)|+|\partial_x^3R_j(t,x)|)\\
&\lesssim_{\omega_1,..,\omega_K,|c_1|,..,|c_K|}\delta_m(t,x):=\sum_{j\neq m}e^{\frac{-h_j}{2}|x-c_jt|}.
\end{align*}
Recall that
\[
v_{*}=\inf_{j\neq k}h_j|c_j-c_k|.
\]
We have
\begin{align*}
&|(R-R_m)(t,x)|+|\partial_x(R-R_m)(t,x)|+|\partial_x^2(R-R_m)(t,x)|+|\partial_x^3(R-R_m)(t,x)| \lesssim \delta_m(t,x)\\
&\lesssim e^{-\frac{1}{4}v_{*}t}.
\end{align*}
We see that $f,g,r$ are polynomials in $R$, $\partial_xR$, $\partial_x^2R$, $\partial_x^3R$, $\partial_x\overline{R}$ and $\partial_x^2\overline{R}$. Denote
\[
A=\sup_{|u|+|\partial_xu|+|\partial_x^2u|+|\partial_x^3u| \leq \sum_j\norm{R_j}_{H^4}}(|df|+|dg|+|dr|).
\]
We have
\begin{align*}
&|\chi|+|\partial_x\chi|+|\partial_x^2\chi|\\
&\leq |f(R,\overline{R}, \partial_xR)-f_{R_m,\partial_x\overline{R}_m, R_m}|+|g(R,\overline{R},\partial_xR,..)-g(R_m,\overline{R}_m,\partial_xR_m,..)|\\
&\quad +|r(R,\partial_xR,..,\partial_x^3R,\overline{R},..)-r(R_m,\partial_xR_m,..,\partial_x^3R_m,\overline{R}_m,..)|\\
&+\Sigma_{j\neq m}(f(R_j,\overline{R}_j,\partial_xR_j)+g(R_j,\partial_xR_j,\partial_x^2R_j,\overline{R}_j,\partial_x\overline{R}_j)+r(R_j,...,\partial_x^3R_j,\overline{R}_j,...,\partial_x^2\overline{R}_j))\\
&\lesssim A(|R-R_m|+|\partial_x(R-R_m)|+|\partial_x^2(R-R_m)|+|\partial_x^3(R-R_m)|)\\
&\quad+A\Sigma_{j\neq m}(|R_j|+|\partial_xR_j|+|\partial_x^2R_j|+|\partial_x^3R_j|)\\
&\lesssim 2A\Sigma_{j\neq m} (|R_j|+|\partial_xR_j|+|\partial_x^2R_j|+|\partial_x^3R_j|)\\
&\lesssim 2A \delta_m(t,x). 
\end{align*}
In particular,
\begin{equation}
\label{estimate of chi}
\norm{\chi}_{W^{2,\infty}} \lesssim e^{-\frac{1}{4}v_{*}t}.
\end{equation}
Moreover, 
\begin{align*}
\norm{\chi}_{W^{2,1}}&\lesssim \Sigma_j (\norm{|R_j|^{2\sigma}\partial_xR_j}_{L^1}+\norm{\partial_x(|R_j|^{2\sigma}\partial_xR_j)}_{L^1}+\norm{\partial_x^2(|R_j|^{2\sigma}\partial_xR_j)}_{L^1})\\
&\lesssim \Sigma_j (\norm{R_j}_{H^1}^(2\sigma+1)+\norm{R_j}_{H^2}^{2\sigma+1}+\norm{R_j}_{H^3}^{2\sigma+1})<\infty.
\end{align*}
Thus, using H\"older inequality we obtain
\begin{align*}
\norm{\chi}_{H^2}&\lesssim_{\omega_1,..,\omega_K,|c_1|,..,|c_K|} e^{-\frac{1}{8}v_{*}t}.
\end{align*}
It follows that if $t \gg \max\{\omega_1,...,\omega_K,|c_1|,...,|c_K|\}$ is large enough then 
\[
\norm{\chi}_{H^2} \leq e^{-\frac{1}{16}v_{*}t}.
\]
Setting $\lambda=\frac{1}{16}v_{*}$, we obtain the desired result. 
\end{proof}


\subsection{Proof $G(\varphi,v)=Q(\varphi,v)$}
Let $G(\varphi,v)$ be defined as in \eqref{eq of G(varphi,v)} and $Q$ be defined as in \eqref{eq of Q}. Then we have the following result.
\begin{lemma}\label{equality of G and Q}
Let $v=\partial_x\varphi-\frac{i}{2}|\varphi|^2\varphi$. Then the following equality holds:
\[
G(\varphi,v)=Q(\varphi,v).
\]
\end{lemma}
\begin{proof}
We have
\begin{align*}
P(\varphi,v)&=i\sigma|\varphi|^{2(\sigma-1)}\varphi^2\overline{v}-\sigma(\sigma-1)\varphi\int_{-\infty}^x|\varphi|^{2(\sigma-2)}\Im(v^2\overline{\varphi}^2)\,dy,\\
Q(\varphi,v)&=-i\sigma|\varphi|^{2(\sigma-1)}v^2\overline{\varphi}-\sigma(\sigma-1)v\int_{-\infty}^x|\varphi|^{2(\sigma-2)}\Im(v^2\overline{\varphi}^2)\,dy\\
G(\varphi,v)&=\partial_xP(\varphi,v)-\frac{i}{2}(\sigma+1)|\varphi|^{2\sigma}P(\varphi,v)\\
&\quad+\frac{i}{2}\sigma|\varphi|^{2(\sigma-1)}\varphi^2\overline{P(\varphi,v)}-i\sigma |\varphi|^{2(\sigma-1)}\varphi^2\partial_x^2\overline{\varphi}\\
&\quad-\frac{i}{2}\left[(\sigma+1)\partial_x\varphi\partial_x(|\varphi|^{2\sigma})+\sigma(\sigma+1)|\partial_x\varphi|^2|\varphi|^{2(\sigma-1)}\varphi\right.\\
&\quad\left.+\sigma(\sigma-1)(\partial_x\overline{\varphi})^2|\varphi|^{2(\sigma-2)}\varphi^3\right]. 
\end{align*}
The term contains $\int_{-\infty}^x|\varphi|^{2(\sigma-2)}\Im(v^2\overline{\varphi}^2)\,dy$ in the expression of $G(\varphi,v)$ is the following.
\begin{align*}
&-\sigma(\sigma-1)\partial_x\varphi\int_{-\infty}^x|\varphi|^{2(\sigma-2)}\Im(v^2\overline{\varphi}^2)\,dy\\
&\quad-\frac{i}{2}(\sigma+1)|\varphi|^{2\sigma}(-1)\sigma(\sigma-1)\varphi\int_{-\infty}^x|\varphi|^{2(\sigma-2)}\Im(v^2\overline{\varphi}^2)\,dy\\
&\quad+\frac{i}{2}\sigma|\varphi|^{2(\sigma-1)}\varphi^2(-1)\sigma(\sigma-1)\overline{\varphi}\int_{-\infty}^x|\varphi|^{2(\sigma-2)}\Im(v^2\overline{\varphi}^2)\,dy\\
&=-\sigma(\sigma-1)\int_{-\infty}^x|\varphi|^{2(\sigma-2)}\Im(v^2\overline{\varphi}^2)\,dy\left(\partial_x\varphi-\frac{i}{2}(\sigma+1)|\varphi|^{2\sigma}\varphi+\frac{i}{2}\sigma|\varphi|^{2\sigma}\varphi\right)\\
&=-\sigma(\sigma-1)\int_{-\infty}^x|\varphi|^{2(\sigma-2)}\Im(v^2\overline{\varphi}^2)\,dy\left(\partial_x\varphi-\frac{i}{2}|\varphi|^{2\sigma}\varphi\right)\\
&=-\sigma(\sigma-1)v\int_{-\infty}^x|\varphi|^{2(\sigma-2)}\Im(v^2\overline{\varphi}^2)\,dy,
\end{align*}
which equals to the term contains $\int_{-\infty}^x|\varphi|^{2(\sigma-2)}\Im(v^2\overline{\varphi}^2)\,dy$ in the expression of $Q(\varphi,v)$. We only need to check the equality of the remaining terms. The remaining terms of $G(\varphi,v)$ is the following.
\begin{align}
&i\sigma\partial_x(|\varphi|^{2(\sigma-1)}\varphi^2\overline{v})-\sigma(\sigma-1)|\varphi|^{2(\sigma-2)}\varphi\Im(v^2\overline{\varphi}^2)\nonumber\\
&-\frac{i}{2}(\sigma+1)|\varphi|^{2\sigma}(i\sigma|\varphi|^{2(\sigma-1)}\varphi^2\overline{v})\nonumber\\
&\quad+\frac{i}{2}\sigma|\varphi|^{2(\sigma-1)}\varphi^2(-i\sigma|\varphi|^{2(\sigma-1)}\overline{\varphi}^2v) -i\sigma|\varphi|^{2(\sigma-1)}\varphi^2\partial_x^2\overline{\varphi} \label{term15}\\
&\quad-\frac{i}{2}\left[(\sigma+1)\partial_x\varphi\partial_x(|\varphi|^{2\sigma})+\sigma(\sigma+1)|\partial_x\varphi|^2|\varphi|^{2(\sigma-1)}\varphi\right.\nonumber\\
&\quad\left.+\sigma(\sigma-1)(\partial_x\overline{\varphi})^2|\varphi|^{2(\sigma-2)}\varphi^3\right].\label{term16}
\end{align}     
Noting that $\partial_x(|\varphi|^2)=2\Re(v\overline{\varphi})$ and $v=\partial_x\varphi-\frac{i}{2}|\varphi|^{2\sigma}\varphi$, we have 
\begin{align*}
&\text{ the term \eqref{term15}}\\
&= i\sigma\partial_x(|\varphi|^{2(\sigma-1)})\varphi^2\overline{v}+i\sigma|\varphi|^{2(\sigma-1)}2\varphi\partial_x\varphi\overline{v}+i\sigma|\varphi|^{2(\sigma-1)}\varphi^2\partial_x\overline{v}\\
&\quad-\sigma(\sigma-1)|\varphi|^{2(\sigma-2)}\varphi 2\Re(v\overline{\varphi})\Im(v\overline{\varphi})+\frac{1}{2}\sigma |\varphi|^{4\sigma-2}\varphi^2\overline{v}\\
&\quad+\sigma^2|\varphi|^{4\sigma-2}\varphi\Re(\varphi\overline{v})-i\sigma|\varphi|^{2(\sigma-1)}\varphi^2\partial_x^2\overline{\varphi}\\
&=2i\sigma(\sigma-1)|\varphi|^{2(\sigma-2)}\Re(v\overline{\varphi})\varphi^2\overline{v}+2i\sigma|\varphi|^{2(\sigma-1)}\varphi\partial_x\overline{v}+i\sigma|\varphi|^{2(\sigma-1)}\varphi^2\partial_x(\overline{v}-\overline{\partial_x\varphi})\\
&\quad-2\sigma(\sigma-1)|\varphi|^{2(\sigma-2)}\varphi\Re(v\overline{\varphi})\Im(v\overline{\varphi})+\frac{1}{2}\sigma|\varphi|^{4\sigma-2}\varphi^2\overline{v}+\sigma^2|\varphi|^{4\sigma-2}\varphi\Re(\varphi\overline{v})\\
&=2\sigma(\sigma-1)|\varphi|^{2(\sigma-2)}\Re(v\overline{\varphi})\varphi(i\varphi\overline{v}-\Im(v\overline{\varphi}))+2i\sigma|\varphi|^{2(\sigma-1)}\varphi\partial_x\overline{v}\\
&\quad+i\sigma|\varphi|^{2(\sigma-1)}\varphi^2\partial_x\left(\frac{i}{2}|\varphi|^{2\sigma}\overline{\varphi}\right)+\frac{1}{2}\sigma|\varphi|^{4\sigma-2}\varphi^2\overline{v}+\sigma^2|\varphi|^{4\sigma-2}\varphi\Re(\varphi\overline{v})\\
&=2i\sigma(\sigma-1)|\varphi|^{2(\sigma-2)}\varphi(\Re(v\overline{\varphi}))^2+2i\sigma|\varphi|^{2(\sigma-1)}\varphi\partial_x\varphi\overline{v}\\
&\quad-\frac{1}{2}\sigma|\varphi|^{2(\sigma-1)}\varphi^2(2\sigma|\varphi|^{2(\sigma-1)}\Re(v\overline{\varphi})+|\varphi|^{2\sigma}\partial_x\overline{\varphi})\\
&\quad+\frac{1}{2}\sigma|\varphi|^{4\sigma-2}\varphi^2\overline{v}+\sigma^2|\varphi|^{4\sigma-2}\varphi\Re(\varphi\overline{v})\\
&=2i\sigma(\sigma-1)|\varphi|^{2(\sigma-2)}\varphi(\Re(v\overline{\varphi}))^2+2i\sigma|\varphi|^{2(\sigma-1)}\varphi\partial_x\varphi\overline{v}\\
&\quad-\frac{1}{2}\sigma|\varphi|^{4\sigma-2}\varphi^2\partial_x\overline{\varphi}+\frac{1}{2}\sigma|\varphi|^{4\sigma-2}\varphi^2\overline{v}\\
&=2i\sigma(\sigma-1)|\varphi|^{2(\sigma-2)}\varphi(\Re(v\overline{\varphi}))^2+2i\sigma|\varphi|^{2(\sigma-1)}\varphi\partial_x\varphi\overline{v}+\frac{1}{2}\sigma|\varphi|^{4\sigma-2}\varphi^2(\overline{v}-\partial_x\overline{\varphi})\\
&=2i\sigma(\sigma-1)|\varphi|^{2(\sigma-2)}\varphi(\Re(v\overline{\varphi}))^2+2i\sigma|\varphi|^{2(\sigma-1)}\varphi\partial_x\varphi\overline{v}+\frac{i}{4}\sigma|\varphi|^{6\sigma}\varphi.
\end{align*}
Moreover, using $\Re(\partial_x\varphi\overline{\varphi})=\Re(v\overline{\varphi})$ we have
\begin{align*}
&\text{ the term \eqref{term16}}\\
&= \frac{-i}{2}\left[\sigma(\sigma+1)|\partial_x\varphi|^2|\varphi|^{2(\sigma-1)}\varphi+\sigma(\sigma+1)|\varphi|^{2(\sigma-1)}\partial_x\varphi(\partial_x\overline{\varphi}\varphi+\partial_x\varphi\overline{\varphi})\right.\\
&\quad\left.+\sigma(\sigma-1)(\partial\overline{\varphi})^2|\varphi|^{2(\sigma-2)}\varphi^3\right]\\
&=\frac{-i}{2}\left[2\sigma|\partial\varphi|^2|\varphi|^{2(\sigma-1)}\varphi +\sigma(\sigma-1)|\varphi|^{2(\sigma-2)}\partial_x\overline{\varphi}\varphi^2(\partial_x\varphi\overline{\varphi}+\partial_x\overline{\varphi}\varphi)\right.\\
&\quad\left.+2\sigma(\sigma+1)|\varphi|^{2(\sigma-1)}\partial_x\varphi\Re(v\overline{\varphi})\right]\\
&=\frac{-i}{2}\left[2\sigma|\partial\varphi|^2|\varphi|^{2(\sigma-1)}\varphi+2\sigma(\sigma-1)|\varphi|^{2(\sigma-2)}\partial_x\overline{\varphi}\varphi^2\Re(v\overline{\varphi})\right.\\
&\quad\left.+2\sigma(\sigma+1)|\varphi|^{2(\sigma-1)}\partial_x\varphi\Re(v\overline{\varphi})\right]\\
&=-i\left[\sigma|\partial\varphi|^2|\varphi|^{2(\sigma-1)}\varphi+\sigma(\sigma-1)|\varphi|^{2(\sigma-2)}\partial_x\overline{\varphi}\varphi^2\Re(v\overline{\varphi})\right.\\
&\quad\left.+\sigma(\sigma+1)|\varphi|^{2(\sigma-1)}\partial_x\varphi\Re(v\overline{\varphi})\right]\\
&=-i\left[\sigma|\partial\varphi|^2|\varphi|^{2(\sigma-1)}\varphi+\sigma(\sigma-1)|\varphi|^{2(\sigma-2)}\Re(v\overline{\varphi})\varphi(\partial_x\overline{\varphi}\varphi+\partial_x\varphi\overline{\varphi})\right.\\
&\quad\left.+2\sigma|\varphi|^{2(\sigma-1)}\partial_x\varphi\Re(v\overline{\varphi})\right]\\
&=-i\left[\sigma|\partial\varphi|^2|\varphi|^{2(\sigma-1)}\varphi+2\sigma(\sigma-1)|\varphi|^{2(\sigma-2)}(\Re(v\overline{\varphi}))^2\varphi\right]\\
&=-2i\sigma(\sigma-1)|\varphi|^{2(\sigma-2)}\varphi(\Re(v\overline{\varphi}))^2\\
&\quad-i\sigma|\partial_x\varphi|^2|\varphi|^{2(\sigma-1)}\varphi-2i\sigma|\varphi|^{2(\sigma-1)}\partial_x\varphi\Re(v\overline{\varphi}).
\end{align*}
Combining the above expressions we obtain
\begin{align*}
&\text{the remaining term of $G(\varphi,v)$}\\
&=2i\sigma|\varphi|^{2(\sigma-1)}\varphi\partial_x\varphi\overline{v}+\frac{i}{4}\sigma|\varphi|^{6\sigma}\varphi-i\sigma|\partial_x\varphi|^2|\varphi|^{2(\sigma-1)}\varphi-2i\sigma|\varphi|^{2(\sigma-1)}\partial_x\varphi\Re(v\overline{\varphi})\\
&=2i\sigma|\varphi|^{2(\sigma-1)}\partial_x\varphi(\varphi\overline{v}-\Re(v\overline{\varphi}))+\frac{i}{4}\sigma|\varphi|^{6\sigma}\varphi-i\sigma|\partial_x\varphi|^2|\varphi|^{2(\sigma-1)}\varphi\\
&=-2\sigma|\varphi|^{2(\sigma-1)}\partial_x\varphi\Im(\varphi\overline{v})+\frac{i}{4}\sigma|\varphi|^{6\sigma}\varphi-i\sigma|\partial_x\varphi|^2|\varphi|^{2(\sigma-1)}\varphi\\
&=-\sigma|\varphi|^{2(\sigma-1)}\partial_x\varphi(2\Im(\varphi\overline{v})+i\partial_x\overline{\varphi}\varphi)+\frac{i}{4}\sigma|\varphi|^{6\sigma}\varphi\\
&=-\sigma|\varphi|^{2(\sigma-1)}\partial_x\varphi(2\Im(\varphi\partial_x\overline{\varphi})+|\varphi|^{2\sigma+2}+i\Re(\varphi\partial_x\overline{\varphi})-\Im(\varphi\partial_x\overline{\varphi}))+\frac{i}{4}\sigma|\varphi|^{6\sigma}\varphi\\
&=-\sigma|\varphi|^{2(\sigma-1)}\partial_x\varphi(|\varphi|^{2\sigma+2}+i\overline{\varphi}\partial_x\varphi)+\frac{i}{4}\sigma|\varphi|^{6\sigma}\varphi\\
&=-i\sigma|\varphi|^{2(\sigma-1)}\overline{\varphi}(\partial_x\varphi)^2-\sigma|\varphi|^{4\sigma}\partial_x\varphi+\frac{i}{4}\sigma|\varphi|^{6\sigma}\varphi\\
&=-i\sigma|\varphi|^{2(\sigma-1)}\overline{\varphi}\left(v+\frac{i}{2}|\varphi|^{2\sigma}\varphi\right)^2-\sigma|\varphi|^{4\sigma}\left(v+\frac{i}{2}|\varphi|^{2\sigma}\varphi\right)+\frac{i}{4}\sigma|\varphi|^{6\sigma}\varphi\\
&=-i\sigma|\varphi|^{2(\sigma-1)}\overline{\varphi}v^2.
\end{align*} 
This is exactly the remaining terms of $Q(\varphi,v)$. Thus, $G(\varphi,v)=Q(\varphi,v)$. 
\end{proof}

\subsection{Existence of a solution of the system}
In this section, using similar arguments as in \cite{CoDoTs15,CoTs14}, we prove the existence of a solution of \eqref{system new}. For convenience, we recall the equation:
\begin{equation}
\label{duhamel form of system}
\eta(t)=i\int_t^{\infty}S(t-s)[f(W+\eta)-f(W)+H](s)\,ds,
\end{equation}
where
\begin{align*}
W&=(h,k),\\
H&=e^{-\lambda t}(m,n),\\
f(\varphi,\psi)&=(P(\varphi,\psi),Q(\varphi,\psi)).
\end{align*}
We have the following lemma.
\begin{lemma}\label{existence slution of system}
Let $H=H(t,x):[0,\infty)\times \R \rightarrow \C^2$, $W=W(t,x):[0,\infty)\times \R \rightarrow \C^2$ be given vector functions which satisfy for some $C_1>0$, $C_2>0$, $\lambda>0$, $T_0 \geq 0$:
\begin{align}
\norm{W(t)}_{L^{\infty}\times L^{\infty}}+e^{\lambda t}\norm{H(t)}_{L^2 \times L^2} &\leq C_1, \quad \forall t \geq T_0, \label{estimateWHinL2normb1}\\
\norm{\partial W(t)}_{L^2 \times L^2}+\norm{\partial W(t)}_{L^{\infty} \times L^{\infty}}+e^{\lambda t}\norm{\partial H(t)}_{L^2 \times L^2} &\leq C_2, \quad \forall t \geq T_0. \label{estimateWHinderivativenormb1}
\end{align} 
Consider equation \eqref{duhamel form of system}. There exists a constant $\lambda_{*}$ independent of $C_2$ such that if $\lambda \geq \lambda_{*}$ then there exists a unique solution $\eta$ of \eqref{duhamel form of system} on $[T_0,\infty) \times \R$ satisfying
\[
e^{\lambda t}\norm{\eta}_{S([t,\infty)) \times S([t,\infty))}+e^{\lambda t}\norm{\partial \eta}_{S([t,\infty)) \times S([t,\infty))} \leq 1, \quad \forall t \geq T_0.
\]
\end{lemma}

\begin{proof}
We rewrite \eqref{duhamel form of system} by $\eta=\Phi\eta$. We show that, for $\lambda$ large enough, $\Phi$ is a contraction map in the following ball
\[
B=\left\{\eta:\norm{\eta}_X:=e^{\lambda t}\norm{\eta}_{S([t,\infty))\times S([t,\infty))}+e^{\lambda t}\norm{\partial_x\eta}_{S([t,\infty))\times S([t,\infty))}\leq 1\right\}.
\]
We will use condition $\lambda \gg 1$ in the proof without specifying it. \\
\textbf{Step 1. Proof $\Phi$ maps $B$ into $B$}\\
Let $t \geq T_0$, $\eta=(\eta_1,\eta_2) \in B$, $W=(w_1,w_2)$ and $H=(h_1,h_2)$. By Strichartz estimates, we have
\begin{align}
\norm{\Phi\eta}_{S([t,\infty))\times S([t,\infty))}&\lesssim \norm{f(W+\eta)-f(W)}_{N([t,\infty))\times N([t,\infty))},\label{eq195}\\
&\quad +\norm{H}_{L^1_{\tau}L^2_x([t,\infty))\times L^1_{\tau}L^2_x([t,\infty))}.\label{eq295}
\end{align}
For \eqref{eq295}, using \eqref{estimateWHinL2normb1}, we have
\begin{align}\label{eq395}
\norm{H}_{L^1_{\tau}L^2_x([t,\infty))\times L^1_{\tau}L^2_x([t,\infty))}&=\norm{h_1}_{L^1_{\tau}L^2_x([t,\infty))}+\norm{h_2}_{L^1_{\tau}L^2_x([t,\infty))}\nonumber\\
&\lesssim \int_t^{\infty}e^{-\lambda \tau}\,d\tau\leq \frac{1}{\lambda}e^{-\lambda t}<\frac{1}{10}e^{-\lambda t}.
\end{align}
For \eqref{eq195}, we have
\begin{align}
&|P(W+\eta)-P(W)|\nonumber\\
&=|P(w_1+\eta_1,w_2+\eta_2)-P(w_1,w_2)|\nonumber\\
&\lesssim \left||w_1+\eta_1|^{2\sigma-1)}(w_1+\eta_1)^2\overline{w_2+\eta_2}-|w_1|^{2(\sigma-1)}w_1^2\overline{w_2}\right|\label{eq4.1}\\
&+\left|(w_1+\eta_1)\int_{-\infty}^x|w_1+\eta_1|^{2(\sigma-2)}\Im((w_2+\eta_2)^2(\overline{w_1}+\overline{\eta_1})^2)\right.\nonumber\\
&\quad\left.-w_1\int_{-\infty}^x|w_1|^{2(\sigma-2)}\Im(w_2^2\overline{\eta_1}^2)\right|.\label{eq4.2}
\end{align}
Using the assumption $\sigma\geq \frac{5}{2}$ and the inequality \eqref{eq 11111} we have
\begin{align*}
&\text{ the term \eqref{eq4.1} }\\
&\lesssim \left|||w_1+\eta_1|^{2(\sigma-1)}-|w_1|^{2(\sigma-1)}||w_1+\eta_1|^2|w_2+\eta_2|\right|\\
&\quad+\left||w_1|^{2(\sigma-1)}|(w_1+\eta_1)^2-w_1^2||w_2+\eta_2|\right|+\left||w_1|^{2(\sigma-1)}|w_1|^2|\eta_2|\right|\\
&\lesssim (|\eta_1|^{2(\sigma-1)}+|\eta_1||w_1|^{2(\sigma-1)-1})(|W|+|\eta|)^3\\
&\quad+|w_1|^{2(\sigma-1)}(|w_1||\eta_1|+|\eta_1|^2)|w_2+\eta_2|+|w_1|^{2\sigma}|\eta_2|\\
&\lesssim (|\eta|^{2(\sigma-1)}+|\eta||W|^{2(\sigma-1)-1})(|W|^3+|\eta|^3)\\
&\quad+|W|^{2(\sigma-1)}(|W||\eta|+|\eta|^2)(|W|+|\eta|)+|W|^{2\sigma}|\eta|\\
&\lesssim |\eta|(|\eta|^{2\sigma-3}+|W|^{2\sigma-3})(|\eta|^3+|W|^3)+|\eta||W|^{2(\sigma-1)}(|W|^2+|\eta|^2)+|W|^{2\sigma}|\eta|\\
&\lesssim |\eta|(|\eta|^{2\sigma}+|W|^{2\sigma})+|\eta||W|^{2\sigma}+|\eta|^3|W|^{2(\sigma-1)}+|W|^{2\sigma}|\eta|\\
&\lesssim |\eta|^{2\sigma+1}+|\eta||W|^{2\sigma}.
\end{align*}
Moreover,
\begin{align*}
&\text{ the term \eqref{eq4.2}}\\
&\lesssim |\eta_1|\int_{-\infty}^x|w_1+\eta_1|^{2(\sigma-2)}|w_2+\eta_2|^2|w_1+\eta_1|^2\,dy\\
&\quad +|w_1|\int_{-\infty}^x(|w_1+\eta_1|^{2(\sigma-2)}-|w_1|^{2(\sigma-2)})|w_2+\eta_2|^2|w_1+\eta_1|^2\,dy\\
&\quad +|w_1|\int_{-\infty}^x|w_1|^{2(\sigma-2)}|\Im((w_2+\eta_2)^2-w_2^2)(\overline{w_1}+\overline{\eta_1})^2|\,dy\\
&\quad +|w_1|\int_{-\infty}^x|w_1|^{2(\sigma-2)}|\Im(w_2^2((\overline{w_1}+\eta_1)^2-\overline{\eta_1}^2))|\,dy\\
&\lesssim |\eta|\int_{-\infty}^x|W|^{2\sigma}+|\eta|^{2\sigma}\, dy+|W|\int_{-\infty}^x(|\eta_1|^{2(\sigma-2)}+|\eta_1||w_1|^{2\sigma-5})(|W|^4+|\eta|^4)\,dy\\
&\quad+|W|\int_{-\infty}^x|W|^{2(\sigma-2)}(|\eta_2|^2+|w_2||\eta_2|)(|W|^2+|\eta|^2)\,dy \\
&\quad+|W|\int_{-\infty}^x|W|^{2(\sigma-2)}|w_2|^2(|\eta_1|^2+|\eta_1||w_1|)\,dy\\
&\lesssim |\eta|\int_{-\infty}^x|W|^{2\sigma}+|\eta|^{2\sigma}\, dy+|W|\int_{-\infty}^x|\eta|(|W|^{2\sigma}+|\eta|^{2\sigma})\,dy\\
&+|W|\int_{-\infty}^x|W|^{2(\sigma-2)}|\eta|(|W|^3+|\eta|^3)\,dy+|W|\int_{-\infty}^x|W|^{2(\sigma-2)}|W|^2|\eta|(|W|+|\eta|)\,dy\\
&\lesssim |\eta|\int_{-\infty}^x|W|^{2\sigma}+|\eta|^{2\sigma}\,dy+|W|\int_{-\infty}^x|\eta||W|^{2\sigma-1}+|\eta|^{2\sigma}\,dy.
\end{align*}
Thus, we obtain
\begin{align*}
&|P(W+\eta)-P(W)|\\
&\lesssim |\eta|^{2\sigma+1}+|\eta||W|^{2\sigma}+|\eta|\int_{-\infty}^x|W|^{2\sigma}+|\eta|^{2\sigma}\,dy+|W|\int_{-\infty}^x|\eta||W|^{2\sigma-1}+|\eta|^{2\sigma}\,dy.
\end{align*}
Similarly,
\begin{align*}
&|Q(W+\eta)-Q(W)|\\
&\lesssim |\eta|^{2\sigma+1}+|\eta||W|^{2\sigma}+|\eta|\int_{-\infty}^x|W|^{2\sigma}+|\eta|^{2\sigma}\,dy+|W|\int_{-\infty}^x|\eta||W|^{2\sigma-1}+|\eta|^{2\sigma}\,dy.
\end{align*}
Hence, using $\sigma \geq \frac{5}{2}$, we have:
\begin{align*}
&\norm{f(W+\eta)-f(W)}_{N([t,\infty))\times N([t,\infty))}\\
&\lesssim \norm{P(W+\eta)-P(W)}_{L^1_{\tau}L^2_x([t,\infty))}+\norm{Q(W+\eta)-Q(W)}_{L^1_{\tau}L^2_x([t,\infty))}\\
&\lesssim \norm{|\eta|^{2\sigma+1}}_{L^1_{\tau}L^2_x([t,\infty))}+\norm{|\eta|\int_{-\infty}^x|W|^{2\sigma}+|\eta|^{2\sigma}\,dy}_{L^1_{\tau}L^2_x([t,\infty))}\\
&\quad+\norm{|W|\int_{-\infty}^x|\eta||W|^{2\sigma-1}+|\eta|^{2\sigma}\,dy}_{L^1_{\tau}L^2_x([t,\infty))}\\
&\lesssim \norm{|\eta|}_{L^{\infty}L^2_x([t,\infty))} \norm{|\eta|}^4_{L^4_{\tau}L^{\infty}_x([t,\infty))}\\
&\quad+\norm{|\eta|}_{L^1_{\tau}L^2_x([t,\infty))}\left\lVert\int_{-\infty}^x|W|^{2\sigma}+|\eta|^{2\sigma}\,dy\right\rVert_{L^{\infty}_{\tau}L^{\infty}_x([t,\infty))}\\
&\quad +\norm{|W|}_{L^{\infty}_{\tau}L^2_x([t,\infty))}\norm{\int_{-\infty}^x|\eta||W|^{2\sigma-1}+|\eta|^{2\sigma}\,dy}_{L^1_{\tau}L^{\infty}_x([t,\infty))}\\
&\lesssim e^{-5\lambda t}+ \norm{|\eta|}_{L^1_{\tau}L^2_x([t,\infty))}\norm{|W|^{2\sigma}+|\eta|^{2\sigma}}_{L^{\infty}_{\tau}L^1_x}\\
&\quad+\norm{W}_{L^{\infty}_tL^2_x}\norm{\eta}_{L^1_{\tau}L^2_x([t,\infty))}\norm{|W|^{2\sigma-1}+|\eta|^{2\sigma-1}}_{L^{\infty}_{\tau}L^2_x([t,\infty))}\\
&\lesssim e^{-5\lambda t}+ \norm{|\eta|}_{L^1_{\tau}L^2_x([t,\infty))}=e^{-5\lambda t}+\int_{t}^{\infty}e^{-\lambda \tau}d\tau\\
&\lesssim e^{-5\lambda t}+\frac{1}{\lambda}e^{-\lambda t}<\frac{1}{10}e^{-\lambda t},
\end{align*}
Combining with \eqref{eq395} and \eqref{eq195}, \eqref{eq295} we obtain
\begin{align}\label{estimate 1}
\norm{\Phi\eta}_{S([t,\infty))\times S([t,\infty))} <\frac{1}{5}e^{-\lambda t}.
\end{align}
We have
\begin{align}
\norm{\partial_x\Phi\eta}_{S([t,\infty))\times S([t,\infty))}&\lesssim \norm{\partial_x(f(W+\eta)-f(W))}_{N([t,\infty))\times N([t,\infty))}\label{eq 100}\\
&\quad +\norm{\partial_xH}_{L^1_{\tau}L^2_x([t,\infty))\times L^1_{\tau}L^2_x([t,\infty))} \label{eq 101}.
\end{align}
For \eqref{eq 101}, using \eqref{estimateWHinderivativenormb1} we have
\begin{equation}
\label{eq 102}
\norm{\partial_xH}_{L^1_{\tau}L^2_x([t,\infty))\times L^1_{\tau}L^2_x([t,\infty))}\lesssim \int_t^{\infty} e^{-\lambda \tau}\,d\tau=\frac{1}{\lambda}e^{-\lambda t}<\frac{1}{10} e^{-\lambda t},
\end{equation}
For \eqref{eq 100}, we have
\begin{align*}
\norm{\partial_x(f(W+\eta)-f(W))}_{N([t,\infty))\times N([t,\infty))}&=\norm{\partial_x(P(W+\eta)-P(W))}_{N([t,\infty))}\\
&+\norm{\partial_x(Q(W+\eta)-Q(W))}_{N([t,\infty))}.
\end{align*}
Furthermore, 
\begin{align}
&|\partial_x(P(W+\eta)-P(W))|\nonumber\\
&\lesssim |\partial_x(|w_1+\eta_1|^{2(\sigma-1)}(w_1+\eta_1)^2(\overline{w}_2+\overline{\eta}_2)-|w_1|^{2(\sigma-1)}w_1^2\overline{w}_2)|\label{1001}\\
&\quad +\left|\partial_x(w_1+\eta_1)\int_{-\infty}^x|w_1+\eta_1|^{2(\sigma-2)}\Im((w_2+\eta_2)^2(\overline{w}_1+\overline{\eta}_1)^2)\,dy\right.\nonumber\\
&\quad\left.-\partial_xw_1\int_{-\infty}^x|w_1|^{2(\sigma-2)}\Im(w_2^2\overline{w}_1^2)\,dy\right|\label{1002}\\
& \quad+\left|(w_1+\eta_1)|w_1+\eta_1|^{2(\sigma-2)}\Im((w_2+\eta_2)^2(\overline{w}_1+\overline{\eta}_1)^2)\right.\nonumber\\
&\quad\left.-w_1|w_1|^{2(\sigma-2)}\Im(w_2^2\overline{w}_1)\right|.\label{1003}
\end{align}
For \eqref{1001}, we have
\begin{align*}
&\text{ the term \eqref{1001}}\\
&\lesssim (|\eta|+|\eta|^{2\sigma}+|\partial_x\eta|)(|W|+|W|^{2\sigma}+|\eta|+|\eta|^{2\sigma}+|\partial_x\eta|)
\end{align*} 
Thus,
\begin{align*}
\norm{\text{the term \eqref{1001}}}_{L^1_{\tau}L^2_x([t,\infty))}&\lesssim \norm{|\eta|+|\partial\eta|}_{L^1_{\tau}L^2_x}\lesssim \frac{1}{\lambda}e^{-\lambda t}<\frac{1}{10}e^{-\lambda t}.
\end{align*}
For \eqref{1002}, using the inequality \eqref{eq 11111}, we have
\begin{align*}
&\norm{\text{ the term \eqref{1002}}}_{L^1_{\tau}L^2_x([t,\infty))}\\
&\lesssim \norm{\partial\eta_1}_{L^1_{\tau}L^2_x([t,\infty))}\norm{\int_{-\infty}^x|w_1+\eta_1|^{2(\sigma-2)}\Im((w_2+\eta_2)^2(\overline{w}_1+\overline{\eta}_1)^2)\,dy}_{L^{\infty}_tL^{\infty}_x}\\
&\quad+\norm{\partial_xw_1}_{L^{\infty}_tL^2_x}\times\nonumber\\&\quad\times\left\lVert\int_{-\infty}^x(|w_1+\eta_1|^{2(\sigma-2)}\Im((w_2+\eta_2)^2(\overline{w}_1+\overline{\eta}_1)^2)-|w_1|^{2(\sigma-2)}\Im(w_2^2\overline{w}_1^2))\, dy\right\rVert_{L^{1_{\tau}L^{\infty}_x}}\\
&\lesssim \norm{\partial\eta_1}_{L^1_{\tau}L^2_x([t,\infty))}\norm{|w_1+\eta_1|^{2(\sigma-2)}\Im((w_2+\eta_2)^2(\overline{w}_1+\overline{\eta}_1)^2)}_{L^{\infty}_tL^1_x}\\
&\quad+\norm{|w_1+\eta_1|^{2(\sigma-2)}\Im((w_2+\eta_2)^2(\overline{w}_1+\overline{\eta}_1)^2)-|w_1|^{2(\sigma-2)}\Im(w_2^2\overline{w}_1^2)}_{L^1_{\tau}L^1_x}\\
&\lesssim \norm{\partial\eta_1}_{L^1_{\tau}L^2_x([t,\infty))}+\norm{|\eta|}_{L^1_{\tau}L^2_x([t,\infty))}\leq \int_t^{\infty}e^{-\lambda \tau}\,d\tau\lesssim \frac{1}{\lambda}e^{-\lambda t}<\frac{1}{10}e^{-\lambda t},
\end{align*}
For \eqref{1003}, using the inequality \eqref{eq 11111}, we have
\begin{align*}
&\norm{\text{the term \eqref{1003}}}_{L^1_{\tau}L^2_x([t,\infty))}\\
&\lesssim \norm{|\eta|}_{L^1_{\tau}L^2_x([t,\infty))}\\
&\leq \int_t^{\infty}e^{-\lambda \tau}\,d\tau\lesssim \frac{1}{\lambda}e^{-\lambda t}<\frac{1}{10}e^{-\lambda t},
\end{align*}
Combining the above estimates, we obtain 
\begin{align}
&\norm{\partial_x(P(W+\eta)-P(W))}_{N([t,\infty))} \nonumber\\
&\leq \norm{\partial_x(P(W+\eta)-P(W))}_{L^1_{\tau}L^2_x([t,\infty))}\leq \frac{3}{10}e^{-\lambda t}, \label{eq 19}
\end{align}
Similarly, 
\begin{align}\label{eq 20}
\norm{\partial_x(Q(W+\eta)-Q(W))}_{N([t,\infty))}&\leq \frac{3}{10}e^{-\lambda t},
\end{align}
Combining the estimates \eqref{eq 100}, \eqref{eq 101}, \eqref{eq 102}, \eqref{eq 19} and \eqref{eq 20}, we have
\begin{equation}
\label{estimate 2}
\norm{\partial_x\Phi\eta}_{S([t,\infty))\times S([t,\infty))}\leq \frac{7}{10}e^{-\lambda t}.
\end{equation} 
Combining \eqref{estimate 1} with \eqref{estimate 2}, we obtain 
\begin{equation}\label{eq final step 1}
\norm{\Phi\eta}_{S([t,\infty))\times S([t,\infty))}+\norm{\partial_x\Phi\eta}_{S([t,\infty))\times S([t,\infty))}\leq \frac{9}{10}e^{-\lambda t},
\end{equation}
Thus, for $\lambda$ large enough
\[
\norm{\Phi\eta}_X <1.
\]
This implies that $\Phi$ maps $B$ into $B$.\\
\textbf{Step 2. $\Phi$ is a contraction map on B}\\
By using \eqref{estimateWHinL2normb1}, \eqref{estimateWHinderivativenormb1} and a similar estimate of \eqref{eq final step 1}, we can show that, for any $\eta\in B$ and $\kappa \in B$ we have
\[
\norm{\Phi\eta-\Phi\kappa}_X\leq \frac{1}{2}\norm{\eta-\kappa}_X.
\]
for $\lambda$ large enough. From Banach fixed point theorem, there exists a unique solution in $B$ of \eqref{duhamel form of system} and thus a solution of \eqref{system new}. This completes the proof of Lemma \ref{existence slution of system}.
\end{proof}

\section*{Acknowledgement}The author is supported by scholarship of MESR for his phD. This work is supported by the ANR LabEx CIMI (grant ANR-11-LABX-0040) within the French State Programme “Investissements d’Avenir.


\bibliographystyle{abbrv}
\bibliography{bibliotheforpaper4}

\begin{thebibliography}{10}

\bibitem{BaWuXu20}
R.~Bai, Y.~Wu, and J.~Xue.
\newblock Optimal small data scattering for the generalized derivative
  nonlinear {S}chr\"{o}dinger equations.
\newblock {\em J. Differential Equations}, 269(9):6422--6447, 2020.

\bibitem{CoOh06}
M.~Colin and M.~Ohta.
\newblock Stability of solitary waves for derivative nonlinear
  {S}chr\"{o}dinger equation.
\newblock {\em Ann. Inst. H. Poincar\'{e} Anal. Non Lin\'{e}aire},
  23(5):753--764, 2006.

\bibitem{GrShSt87}
M.~Grillakis, J.~Shatah, and W.~Strauss.
\newblock Stability theory of solitary waves in the presence of symmetry. {I}.
\newblock {\em J. Funct. Anal.}, 74(1):160--197, 1987.

\bibitem{GrShSt90}
M.~Grillakis, J.~Shatah, and W.~Strauss.
\newblock Stability theory of solitary waves in the presence of symmetry. {II}.
\newblock {\em J. Funct. Anal.}, 94(2):308--348, 1990.

\bibitem{GuNiWu20}
Z.~Guo, C.~Ning, and Y.~Wu.
\newblock Instability of the solitary wave solutions for the generalized
  derivative nonlinear {S}chr\"{o}dinger equation in the critical frequency
  case.
\newblock {\em Math. Res. Lett.}, 27(2):339--375, 2020.

\bibitem{HaOz16}
M.~Hayashi and T.~Ozawa.
\newblock Well-posedness for a generalized derivative nonlinear
  {S}chr\"{o}dinger equation.
\newblock {\em J. Differential Equations}, 261(10):5424--5445, 2016.

\bibitem{Ka72}
T.~Kato.
\newblock Nonstationary flows of viscous and ideal fluids in {${\bf R}^{3}$}.
\newblock {\em J. Functional Analysis}, 9:296--305, 1972.

\bibitem{KwWu18}
S.~Kwon and Y.~Wu.
\newblock Orbital stability of solitary waves for derivative nonlinear
  {S}chr\"{o}dinger equation.
\newblock {\em J. Anal. Math.}, 135(2):473--486, 2018.

\bibitem{CoDoTs15}
S.~Le~Coz, D.~Li, and T.-P. Tsai.
\newblock Fast-moving finite and infinite trains of solitons for nonlinear
  {S}chr\"{o}dinger equations.
\newblock {\em Proc. Roy. Soc. Edinburgh Sect. A}, 145(6):1251--1282, 2015.

\bibitem{CoTs14}
S.~Le~Coz and T.-P. Tsai.
\newblock Infinite soliton and kink-soliton trains for nonlinear
  {S}chr\"{o}dinger equations.
\newblock {\em Nonlinearity}, 27(11):2689--2709, 2014.

\bibitem{CoWu18}
S.~Le~Coz and Y.~Wu.
\newblock Stability of multisolitons for the derivative nonlinear
  {S}chr\"{o}dinger equation.
\newblock {\em Int. Math. Res. Not. IMRN}, (13):4120--4170, 2018.

\bibitem{LiSiSu13}
X.~Liu, G.~Simpson, and C.~Sulem.
\newblock Stability of solitary waves for a generalized derivative nonlinear
  {S}chr\"{o}dinger equation.
\newblock {\em J. Nonlinear Sci.}, 23(4):557--583, 2013.

\bibitem{MaMeTs02}
Y.~Martel, F.~Merle, and T.-P. Tsai.
\newblock Stability and asymptotic stability in the energy space of the sum of
  {$N$} solitons for subcritical g{K}d{V} equations.
\newblock {\em Comm. Math. Phys.}, 231(2):347--373, 2002.

\bibitem{MaMeTs06}
Y.~Martel, F.~Merle, and T.-P. Tsai.
\newblock Stability in {$H^1$} of the sum of {$K$} solitary waves for some
  nonlinear {S}chr\"{o}dinger equations.
\newblock {\em Duke Math. J.}, 133(3):405--466, 2006.

\bibitem{Oz96}
T.~Ozawa.
\newblock On the nonlinear {S}chr\"{o}dinger equations of derivative type.
\newblock {\em Indiana Univ. Math. J.}, 45(1):137--163, 1996.

\bibitem{Santos15}
G.~d.~N. Santos.
\newblock Existence and uniqueness of solution for a generalized nonlinear
  derivative {S}chr\"{o}dinger equation.
\newblock {\em J. Differential Equations}, 259(5):2030--2060, 2015.

\bibitem{TaXu18}
X.~Tang and G.~Xu.
\newblock Stability of the sum of two solitary waves for (g{DNLS}) in the
  energy space.
\newblock {\em J. Differential Equations}, 264(6):4094--4135, 2018.

\end{thebibliography}

\end{document}